\documentclass[12pt]{amsart}
\usepackage{amssymb,latexsym,amsmath,amsthm,amscd,mathrsfs,tikz-cd}
\usepackage{setspace}
\usepackage{verbatim}
\usepackage[active]{srcltx}
\usepackage[all]{xy} \xyoption{arc}
\usepackage[left=3cm,top=2cm,right=3cm,bottom = 2cm]{geometry}
\usepackage{graphicx}

\theoremstyle{plain}
\newtheorem{thm}{Theorem}[section]

\newtheorem{prop}[thm]{Proposition}
\newtheorem{cor}[thm]{Corollary}

\theoremstyle{definition}
\newtheorem{dfn}[thm]{Definition}

\theoremstyle{remark}
\newtheorem{rmk}[thm]{Remark}

\newcommand{\sM}{\mathscr{M}}

\newcommand{\cF}{\mathcal{F}}

\newcommand{\cL}{\mathcal{L}}

\newcommand{\cO}{\mathcal{O}}

\newcommand{\cS}{\mathcal{S}}

\newcommand{\cV}{\mathcal{V}}
\newcommand{\cW}{\mathcal{W}}

\newcommand{\frakh}{\mathfrak{h}}

\DeclareMathOperator{\uhp}{\mathfrak{h}}
\DeclareMathOperator{\rk}{rk}

\DeclareMathOperator{\Tr}{Tr}

\DeclareMathOperator{\Hom}{Hom}
\DeclareMathOperator{\Spec}{Spec}

\DeclareMathOperator{\Aut}{Aut}

\DeclareMathOperator{\Pic}{Pic}

\DeclareMathOperator{\GL}{GL}

\DeclareMathOperator{\SL}{SL}
\DeclareMathOperator{\PSL}{PSL}

\DeclareMathOperator{\Ind}{Ind}

\newcommand*{\df}{\mathrel{\vcenter{\baselineskip0.5ex \lineskiplimit0pt
                     \hbox{\scriptsize.}\hbox{\scriptsize.}}} =}

\providecommand{\twomat}[4]{\left(\begin{matrix}#1&#2\\#3&#4\end{matrix}\right)}
\providecommand{\stwomat}[4]{\left(\begin{smallmatrix}#1&#2\\#3&#4\end{smallmatrix}\right)}

\providecommand{\stwovec}[2]{\left(\begin{smallmatrix}#1\\#2\end{smallmatrix}\right)}

\newcommand{\QQ}{\mathbf{Q}}

\newcommand{\CC}{\mathbf{C}}

\newcommand{\ZZ}{\mathbf{Z}}
\newcommand{\PP}{\mathbf{P}}
\newcommand{\RR}{\mathbf{R}}

\newcommand{\DD}{\mathbf{D}}

\DeclareMathOperator{\GM}{GM}

\DeclareMathOperator{\End}{End}
\DeclareMathOperator{\Ext}{Ext}

\DeclareMathOperator{\lcm}{lcm}

\begin{document}
\title[Fuchsian groups of genus zero]{Vector bundles and modular forms for Fuchsian groups of genus zero}
\author{Luca Candelori and Cameron Franc}
\date{}

\maketitle

\begin{abstract}
This article lays the foundations for the study of modular forms transforming with respect to representations of Fuchsian groups of genus zero. More precisely, we define \emph{geometrically weighted} graded modules of such modular forms, where the graded structure comes from twisting with all isomorphism classes of line bundles on the corresponding compactified modular curve, and we study their structure by relating it to the structure of vector bundles over orbifold curves of genus zero. We prove that these modules are free whenever the Fuchsian group has at most two elliptic points. For three or more elliptic points, we give explicit constructions of indecomposable vector bundles of rank two over modular orbifold curves, which give rise to non-free modules of geometrically weighted modular forms. 
\end{abstract}

\tableofcontents

\section{Introduction}

The history of vector valued modular forms dates back to Poincar\'e's work on Fuchsian functions and linear differential equations \cite{Poincare1}, \cite{Poincare2}. In recent years, vector valued modular forms have played a main role in the mathematics spawned by the proof of the monstrous moonshine conjecture \cite{Borcherds}. Indeed, these modular forms arise as generating series for characters of rational vertex operator algebras, and thus form an important part of their representation theory --- see \cite{Gannon:book} for a survey. Most of the work in this context has so far focused on the case of level one, that is, vector valued modular forms for $\SL_2(\ZZ)$. The aim of this article is to extend the basic structure theory of vector valued modular forms from the case of level one to a general genus zero Fuchsian group. This more general theory turns out to be much richer, and it is closely related to the classification of vector bundles over orbifold curves of genus zero \cite{GeigleLenzing}, \cite{Crawley-Boevey}.

A key result in the theory of vector valued modular forms for $\SL_2(\ZZ)$ is the free module theorem of Marks-Mason \cite{MarksMason}, which asserts that the module of modular forms associated to a rank $r$ representation of $\SL_2(\ZZ)$ is free of rank $r$ over the ring of scalar valued modular forms of level one. In \cite{CandeloriFranc}, we deduced this result from a splitting principle for vector bundles over the modular orbifold of level one, thus establishing the first geometric link between the study of vector bundles over modular curves and the structure theory of vector valued modular forms. Even for subgroups of $\SL_2(\ZZ)$ of small index the picture quickly becomes more complicated. For example, during the summer of 2016 Geoff Mason observed (private correspondence with the authors) that the free module theorem fails when $\SL_2(\ZZ)$ is replaced by its subgroup $\Gamma^2$ of index two. Mason and Gannon then established similar negative results for a number of other subgroups of small index (unpublished note). On the geometric side, we were able to show that the splitting principle for vector bundles holds for $\Gamma^2$ and other subgroups of small index, thus raising the question of whether the structure of vector bundles over modular orbifolds of higher level has any connection at all with that of vector valued modular forms of the corresponding level. One of the key insights of the present paper is the following: if one is willing to work over a ring of {\em geometrically weighted} modular forms, which is slightly larger than the classical ring of scalar valued modular forms, then the corresponding modules of geometrically weighted modular forms possess useful commutative algebraic properties. For example, we prove a free module theorem for geometrically weighted modular forms over  certain Fuchsian groups, including $\Gamma^2$ --- see Corollary \ref{cor:moduleSplitting} and the discussion at the end of Subsection \ref{ss:indexTwo}. Around the same time that we established these results, Terry Gannon used a different argument to prove a similarly modified free module theorem for $\Gamma^2$ and a number of other Fuchsian groups (unpublished note).  Also around this time, Richard Gottesman established, as part of his forthcoming PhD thesis, a free module theorem for $\Gamma_0(2)$ using the argument pioneered in \cite{MarksMason}. Over $\Gamma_0(2)$ the ring of geometrically weighted modular forms is equal to that of classical modular forms, so this case can also be treated using the geometric approach of the present paper.

These results point to a need for a general study of vector valued modular forms on arbitrary Fuchsian groups. To simplify matters we restrict in this paper to Fuchsian subgroups of $\PSL_2(\RR)$ of the first kind, with finite covolume and finitely many cusps. Further, we restrict to subgroups with the property that the corresponding Riemann surface they define is of genus zero when the cusps are included. Such groups are said to be of genus zero, and the corresponding compact orbifolds are called \emph{orbifold lines}, or \emph{orbilines}. Orbifold lines are sufficienly rigid to allow us to reach some nontrivial general conclusions, but not so rigid that they are uninteresting --- for example, orbifold lines are some of the main protagonists in monstrous moonshine \cite{ConwayNorton}, \cite{Borcherds}, and the study of vector bundles on orbifold lines has been connected with the representation theory of certain Kac-Moody algebras \cite{Lenzing}, \cite{GeigleLenzing}, \cite{Meltzer:MemAMS}, \cite{Crawley-Boevey}. Our results can be extended to Fuchsian subgroups of $\SL_2(\RR)$ by replacing orbifold lines with $\mu_2$-gerbes over them, as shown in \cite{CandeloriFranc} for the case of $\SL_2(\ZZ)$. We leave this slight generalization open to further exploration. 

In Section \ref{orbifoldLines} we define orbifold lines and recall some standard facts about them. In Section \ref{indecomposables} we study the structure of  vector bundles over orbifold lines. As is well-known, every vector bundle over $\PP^1$ splits into a sum of line bundles \cite{Groth}. We establish a similar splitting principle for orbifold lines with at most two orbifold points using a proof modeled after the standard cohomological argument for the usual projective line (note that this case includes the compactification of the orbifold quotient $[\PSL_2(\ZZ)\backslash \uhp]$ studied in detail in \cite{CandeloriFranc}). We then recall a general result of Crawley-Boevey (\cite{Crawley-Boevey}, Theorem \ref{thm:CB}) classifying indecomposable vector bundles over orbifold lines in terms of an associated Kac-Moody algebra. In the remainder of the section we focus on an orbifold line with three orbifold points with stabilizers of size $(2,2,n)$, and we give an explicit description of an indecomposable vector bundle of rank $2$ on such an orbifold line (Theorem \ref{thm:ExistenceOfRank2Indec}) that is independent of Crawley-Boevey's existence result. Later in Section \ref{s:examples} we write down the transition functions of such indecomposable bundles explicitly using modular forms.

In Section \ref{coherentSheaves} we study the structure of coherent sheaves over an orbifold line. The key idea is to embed an orbifold line into weighted projective space and then to identify it with a {\em weighted projective line}. The theory of coherent sheaves over weighted projective lines is due to Geigle-Lenzing \cite{GeigleLenzing}, and it follows the usual theory for algebraic curves \cite{SerreFAC}. More precisely, if $X$ denotes an orbifold line with $n+1$ orbifold points, and if $W= \Pic(X)$ denotes the Picard group of isomorphism classes of line bundles on $X$, then we describe a natural projective embedding of $X$ inside a weighted projective space $\PP(W)$ associated to $W$ (Theorem \ref{thm:weightedProjectiveLineIsomorphism}).  This embedding is used in Proposition \ref{prop:equivalenceCoherentSheaves} to give a concrete realization of the category of coherent sheaves on $X$ in terms of a category of sheaves of $W$-graded modules. For any coherent sheaf $\cF$ on $X$ there is a \emph{Serre-type functor} \cite{GeigleLenzing}
\[
  \GM_*(\cF) \df \bigoplus_{x \in W} H^0(X,\cF(x))
\]
to the category of $W$-graded modules over the $W$-graded projective coordinate ring $S(X)$ of the embedding $X \hookrightarrow \PP(W)$. This functor can be used to link the structure of vector bundles over $X$ to that of $W$-graded modules over $S(X)$. For example, Geigle-Lenzing proved (\cite{GeigleLenzing}, 5.1) that if $\cV$ is a vector bundle on $X$, then $\GM_*(\cV)$ is always maximal Cohen-Macaulay over $S(X)$. We also deduce from Theorem \ref{thm:SplittingTheorem} that  $\GM_*(\cV)$ is free over $S(X)$ whenever $X$ has at most two orbifold points (see Corollary \ref{cor:moduleSplitting}). In general, the module $\GM_*(\cV)$ decomposes into a direct sum of indecomposable maximal Cohen-Macauley modules, each indexed by a certain Kac-Moody algebra constructed from $X$. 

When $X$ is obtained by adding the cusps to a quotient $[\Gamma \backslash \uhp]$ for some Fuchsian group $\Gamma \subseteq \PSL_2(\RR)$, the $W$-graded modules $\GM_*(\cV)$ seem to be better substitutes for the usual $\ZZ$-graded modules of modular forms associated to a representation of $\Gamma$. In Section \ref{s:vvmfs} we apply this observation to the study of modular forms for $\Gamma$. We begin by explaining briefly how the constructions of \cite{CandeloriFranc} generalize to arbitrary Fuchsian groups. If $\cV(\rho)$ denotes the canonical extension to $X$ of the local system on $[\Gamma \backslash \uhp]$ associated to a representation $\rho$ of $\Gamma$, then we introduce the module of \emph{geometrically weighted} $\rho$-valued modular forms $\GM(\rho) = \GM_*(\cV(\rho))$. Following the above definition, this is obtained by twisting $\cV(\rho)$ with {\em all} line bundles on $X$ and putting the global sections together into a single module, rather than twisting only with the standard line bundles $\cL_k$ of weight $k$ modular forms. In general this module $\GM(\rho)$ is strictly larger than the usual module $M(\rho)$ of modular forms associated to $\rho$, but it has the advantage that the preceding results on $\GM_*(\cV(\rho))$ apply to it. It is a module over the coordinate ring $S(\Gamma) = S(X)$ of the projective embedding $X \hookrightarrow P(W)$ discussed above, which likewise is typically larger than the usual ring of scalar valued modular forms for $\Gamma$. Using this construction, the classification of vector bundles over orbifold lines directly gives a classification of the modules $\GM(\rho)$. In particular, we deduce that when $\Gamma$ has at most 2 elliptic points, the free module theorem holds for $\GM(\rho)$ (Theorem \ref{thm:FMT}). 

The constructions of the ring $S(\Gamma)$ and the modules $\GM(\rho)$ are geometric in origin, and one would like an automorphic description for them. In certain cases one can use the exact sequence
\[
  0 \to \Pic_0(X) \to \Pic(X) \stackrel{\deg}{\to} \frac 1m \ZZ \to 0,
\] 
where $m$ is the least common multiple of the orders of the stabilzers of the elliptic points for $\Gamma$, to give a convenient automorphic description for $S(\Gamma)$ and $\GM(\rho)$. One interpets $\Pic_0(X)$ as those bundles arising from characters of $\Gamma$ that are trivial on the stabilizers of the cusps, the so-called \emph{cuspidal characters}. This reduces finding an automorphic description for all bundles in $\Pic(X)$ --- thereby giving an automorphic description for $S(\Gamma)$ and $\GM(\rho)$ --- to giving an automorphic description of any single bundle on $X$ of degree $\frac 1m$. In certain lucky cases the weight shifting bundle $\cL_2$ is of appropriate degree, and one finds that $S(\Gamma)$ is the ring generated by the modular forms associated to all the cuspidal characters in all weights, and similarly for $\GM(\rho)$. But in general $\cL_2$ will have too large a degree, and $S(\Gamma)$ and $\GM(\rho)$ contain slightly more than the forms generated by the cuspidal characters. Nevertheless, if $\cL$ denotes a line bundle on $X$ of degree $\frac 1m$, then it pulls back to a trivializable bundle on the Stein space $\uhp$, and this pullback can be described by \emph{some} automorphy factor. It is an open problem to describe a choice of automorphy factor that works for any given Fuchsian group.

In Section \ref{s:examples} we end the paper by describing these constructions in detail for several Fuchsian groups: $\PSL_2(\ZZ)$ and its unique normal subgroups of index two and three, as well as a nonnormal subgroup of index $4$. One of the most surprising findings in these explorations arose from the explicit automorphic construction of indecomposable bundles of rank two on the subgroups of index $3$ and $4$ considered here. The modular construction of these bundles involves certain indecomposable representations that are not unitarizable, and so the corresponding modular forms fall outside the scope of the classical theory of scalar valued modular forms. These vector valued forms are essentially antiderivatives of classical scalar valued modular forms, and in this way the ratios of period integrals of classical scalar valued modular forms between elliptic points arise naturally in the construction of indecomposable vector bundles. This suggests that the CM values of vector valued modular forms associated with certain nonunitary representations may hold some arithmetic interest, although we do not make any general claims in this direction.

As mentioned above, some of our results have been obtained independently by Mason and Gannon. The authors thank them for several useful discussions about the contents of this paper, and for sharing some of their private notes with us.

We end this introduction with a brief discussion about notation regarding Fuchsian groups. In this note Fuchsian groups are subgroups of $\PSL_2(\RR)$, rather than $\SL_2(\RR)$, although we will commit the standard abuse of writing elements of $\PSL_2(\RR)$ as matrices. We retain the notation
\begin{align*}
T &= \twomat 1101, & S &= \twomat{0}{-1}10, & R &= ST,
\end{align*}
of our previous papers, but here $\chi \colon \PSL_2(\ZZ) \to \CC^\times$ denotes the unique character of $\PSL_2(\ZZ)$ determined by $\chi(T) = e^{2\pi i\frac 16}$. In \cite{CandeloriFranc} we used $\chi$ to denote the generating character of $\SL_2(\ZZ)$; in this paper, $\chi$ is the square of that character. That is, $\chi$ denotes the character of $\eta^4$ and it satisfies $\chi(S) = -1$ and $\chi(R) = \zeta^2$, where $\zeta = e^{2\pi i\frac 13}$.

\section{Orbifold lines}
\label{orbifoldLines}

An {\em orbifold curve} $X$ is a compact, connected, complex orbifold of dimension one with finitely many orbifold points $(P_0, \ldots, P_n)$ with non-trivial cyclic stabilizers of orders $(p_0, \ldots, p_n)$, respectively. The {\em genus} of $X$ is the genus of its underlying Riemann surface. An {\em orbifold line} is an orbifold curve of genus 0. Since the only compact, connected Riemann surface of genus 0 is $\PP^1$, the sequence $(p_0,\ldots, p_n)$ determines the orbifold line $X$ uniquely up to re-labeling of the orbifold points $P_i$. We call $(p_0,\ldots, p_n)$ the {\em signature} of $X$. 

Let $\cL$ be a line bundle over $X$. The restriction $\cL|_P$ to an orbifold point $P \in \{ P_0, \ldots, P_n\}$ is a 1-dimensional vector space together with a one-dimensional representation
$$
\mu(\cL,P):\ZZ/p\ZZ \longrightarrow \CC^{\times},
$$
where $p$ is the order of the stabilizer of $P$. This representation is completely determined by the element $\iota(\cL,P)\in \{0,\ldots, p-1\}$ given by $ \mu(\cL,P)(1) = e^{\frac{2\pi i}{p} \iota(\cL,P)}$.

\begin{dfn}
The integer $\iota(\cL,P)\in \{0,\ldots, p-1\}$ is called the {\em isotropy} of $\cL$ at $P$. 
\end{dfn}

Let now $\Pic(X)$ be the group under $\otimes$ of all line bundles over $X$ up to isomorphism. There is an isomorphism   $\Pic(X)\cong \mathrm{Cl}(X)$ with the class group, i.e. all divisors modulo principal divisors. A divisor for an orbifold curve $X$ is a formal linear combination
$$
D = \sum_{P \in X} \frac{a_{P}}{|\mathrm{Stab_X(P)}|}\, P, \quad a_P\in \ZZ,
$$
and a principal divisor is a divisor of a rational function on $X$, where the zeroes and poles of the function are appropriately rescaled by the order of the stabilizers. The {\em degree} of a divisor $D$ is  
$$
\deg(D) = \sum_{P \in X} \frac{a_{P}}{|\mathrm{Stab_X(P)}|} \in \frac{1}{m}\ZZ,
$$
where $m = \mathrm{lcm}(p_0,\ldots,p_n)$. We thus obtain the familiar  degree homomorphism 
$$
\deg: \Pic(X) \longrightarrow \frac{1}{m}\ZZ
$$
by associating to a line bundle  $\cL$ the degree of the divisor of a rational section of $\cL$.

\begin{prop}
\label{prop:picardGroupOfStackyLines}
Let $X$ be an orbifold line of signature $(p_0,\ldots,p_n)$. Then 
\begin{itemize}
\item[(a)]
$\Pic(X)$ is a rank one abelian group generated by $n+1$ elements $x_0,\ldots, x_n$ with relations $$p_0x_0=p_2x_2=\ldots=p_nx_n.$$

\item[(b)] Let $m = \lcm(p_0,\ldots, p_n)$. Then the degree homomorphism induces an exact sequence
$$
0 \rightarrow \Pic_0(X) \longrightarrow \Pic(X) \stackrel{\deg}\longrightarrow \frac{1}{m}\ZZ \rightarrow 0,
$$
that identifies $\Pic_0(X)$ with the torsion subgroup of $\Pic(X)$.
\item[(c)] There is an  isomorphism
\[
  \Pic_0(X)  \cong \frac{\prod_{j=0}^n\ZZ/p_j\ZZ}{\langle (1,1,\ldots, 1)\rangle}.
\]
\end{itemize}
\end{prop}

\begin{proof}
Given $\cL\in \Pic(X)$, mapping $\cL$ to each isotropy $\iota(\cL,P)$, $P\in \{P_0, \ldots, P_n\}$  gives a group homomorphism 
\begin{equation}
\label{eqn:PicRestriction}
\Pic(X)\rightarrow \prod_{i=0}^n \ZZ/p_i\ZZ.
\end{equation}
The kernel of this isomorphism is the Picard group of the underlying Riemann surface, which is isomorphic to $\ZZ$ for an orbifold line. Therefore \eqref{eqn:PicRestriction} gives an exact sequence 
$$
0\rightarrow \ZZ \rightarrow \Pic(X)\rightarrow \prod_{i=0}^n  \ZZ/p_i\ZZ\rightarrow 0,
$$
which proves the claim about the rank of $\Pic(X)$. To get the more precise statement about the relations, let $x_i$ be the class in $\Pic(X)$ corresponding to the line bundle $\cO_{X}(P_i)$. Then each $\cO_{X}(P_i)^{\otimes p_i}$ descends to a line bundle over the underlying Riemann surface $\PP^1$ of degree one. But line bundles over $\PP^1$ of the same degree must necessarily be isomorphic, so part (a) follows. For part (b), note that $\cO_X(P_i)$ has degree $1/p_i$, thus the degree map is surjective, of kernel equal to $ \Pic_0(X)$ since $\frac{1}{m}\ZZ$ has characteristic zero. Now  $\Pic_0(X)$ consists of elements of finite order, so each $\cL \in  \Pic_0(X)$ can be given the structure of an $N$-torsor (i.e. a principal $\ZZ/N\ZZ$-bundle), for some positive integer $N$. These correspond to homomorphisms $\pi_1(X)\rightarrow \Aut(\ZZ/N\ZZ) = \ZZ/N\ZZ$, where $\pi_1(X)$ is the orbifold fundamental group of $X$. Since $X$ is of genus zero, this group has the presentation
\[
  \pi_1(X) = \langle \gamma_0,\ldots, \gamma_n \mid \gamma_j^{p_j} = 1,~ \gamma_0\gamma_1\cdots\gamma_n = 1 \rangle,
\]
thus it is generated by elements of finite order. In particular, one has
\[
\Pic_0(X) = \Hom(\pi_1(X),\CC^\times)\cong \frac{\prod_{j=0}^n \ZZ/p_j\ZZ}{\langle (1,1,\ldots,1)\rangle},
\]
which proves part (c). 
\end{proof}

\begin{rmk}
The identification $\Pic_0(X) \cong \Hom(\pi_1(X),\CC^\times)$ established above is well-known and part of a bigger picture. Narasimhan-Seshadri showed \cite{NarasimhanSeshadri} that stable holomorphic vector bundles of degree zero on a compact Riemann surface are in one-to-one correspondence with irreducible unitary representations of the fundamental group. The case of line bundles and unitary characters of the fundamental group, where the stability condition becomes empty, is even more classical. These results have been extended considerably in the direction of nonabelian Hodge theory, due initially to Hitchin, Donaldson and Simpson -- see \cite{Simpson1} and the references therein. In particular, the Narasimhan-Seshadri result has been generalized \cite{BiswasHogadi}, \cite{Simpson2} to the setting of compact orbifolds discussed in Proposition \ref{prop:picardGroupOfStackyLines}. For line bundles on a compact orbifold of genus zero, where all representations of rank one have finite image, one does not need the full strength of these results, as shown in the proof of Proposition \ref{prop:picardGroupOfStackyLines}. 
\end{rmk}

By a slight abuse of notation (which is standard for the case $X=\PP^1$), for any $x\in \Pic(X)$ we denote by $\cO(x)$ the  corresponding line bundle. For the {\em distinguished} element (of degree one)
$$
c\df p_0x_0 = \ldots = p_nx_n \in \Pic(X)
$$  
we call $\cO(c)$ the {\em distinguished} line bundle, and for the {\em dualizing} element 
$$
\omega\df (n-1)c - \sum_{i=0}^n x_i \in \Pic(X)
$$ 
(of degree $n-1 - \sum_{i=0}^n 1/p_i$) we call $\cO(\omega)$ the {\em dualizing} line bundle. This line bundle is isomorphic to the canonical bundle $\Omega^1_X$, the relative dualizing sheaf of the orbifold line $X$.

For any orbifold line $X$, the abelian group $\Pic(X)$ carries a partial ordering. In particular, since $\Pic(X)/\ZZ c \cong \prod_{i=0}^n \ZZ/p_i\ZZ$, we may write any element of $\Pic(X)$ uniquely as 
$$
x = \sum_{i=0}^n a_i(x) x_i + a(x)c, \quad a_i(x)\in \{0,\ldots,p_i-1\},\quad a(x)\in \ZZ.
$$
We then say that $x\leq y$ if $a_i(x)\leq a_i(y)$ for all $i=0,\ldots,n$ and $a(x)\leq a(y)$. Note that under this partial ordering the degree homomorphism $\Pic(X)\rightarrow \frac{1}{m}\ZZ$ is thus order-preserving.

Let $\cV$ be a vector bundle of rank $r$ over an orbifold line $X$. Restriction to each non-trivial orbifold point $P_i$ gives an $r$-dimensional representation
$$
\mu(\cV,P_i):\ZZ/p_i\ZZ \longrightarrow \GL_r(\CC)
$$
which is entirely determined by the linear transformation $\mu(\cV,P_i)(1)$. This is of finite order, hence diagonalizable, with eigenvalues of the form $e^{\frac{2\pi i}{p_i} \nu_{ij}}$, $\nu_{ij} \in \{0, \ldots, p_i-1\}$,  $j=1,\ldots, r$. 

\begin{dfn}
The integers $\nu_{ij} \in \{0, \ldots, p_i-1\}$, $j=1, \ldots, r$ are called the {\em isotropies} of $\cV$ at $P_i$. The integer
$$
\iota(\cV,P_i)\df \sum_{i=1}^r \nu_{ij} \in \ZZ_{\geq 0}
$$
is called the {\em isotropy trace} of $\cV$ at $P_i$.
\end{dfn}

Let now $\cL$ be a line bundle over an orbifold line $X$. Denote by
$$
\chi(X,\cL) \df \dim H^0(X,\cL) - \dim H^1(X,\cL)
$$
its Euler characteristic. Then the Riemann-Roch theorem for orbifold lines says that
$$
\chi(X,\cL) = \deg \cL + 1  - \left( \sum_{i=0}^n \frac{\iota(\cL,P_i)}{p_i}\right).
$$
As with algebraic curves (\cite{Atiyah}, Lemma 1 and \cite{GeigleLenzing}, Prop. 2.6), any vector bundle $\cV$ of rank $r$ over $X$ has a filtration by sub-bundles
\begin{equation}
\label{VBundleFiltration}
0 = \cV_0 \subseteq \cV_1 \subseteq \cdots \subseteq \cV_r = \cV
\end{equation}
such that each quotient $\cV_i/\cV_{i-1}\cong \cL_i$ is a line bundle. Among all filtrations as in \eqref{VBundleFiltration}, it is possible to choose a {\em maximal} filtration by requiring $\cV_1 = \cL_1$ to be a line sub-bundle of $\cV$ of maximal degree, then $\cL_2$ to be a line sub-bundle of $\cV/\cL_1$ of maximal degree, etc. By choosing any such filtration, the Riemann-Roch theorem extends to vector bundles over $X$  by the formula
$$
\chi(X,\cV) \df \dim H^0(X,\cV) - \dim  H^1(X,\cV) = \deg(V) + \rk(\cV) - \left( \sum_{i=0}^n \frac{\iota(\cV,P_i)}{p_i} \right)
$$
where $\deg(V)\df  \deg(\det(\cV))$. Moreover, the Serre duality theorem for orbifold lines gives a canonical isomorphism
$
H^1(X, \cV) \cong H^0(X, \cV^*(\omega))^*
$
and thus
$$
\dim H^1(X, \cV)  = \dim H^0(X, \cV^*(\omega)). 
$$

\section{Vector bundles over orbifold lines}
\label{indecomposables}

When $X= \PP^1$ the Grothendieck-Bierkhoff Theorem \cite{Groth} says that any vector bundle over $X$ decomposes into a sum of line bundles, i.e. an indecomposable vector bundle over $\PP^1$ is necessarily of rank one. The following Theorem \ref{thm:SplittingTheorem} shows that an analogous result holds for orbifold lines with at most two orbifold points.

\begin{thm}
\label{thm:SplittingTheorem}
Suppose that $X$ is an orbifold line with at most two orbifold points. Then any vector bundle $\cV$ of rank $r$ on $X$ decomposes as a direct sum of line bundles
$$
\cV \cong \bigoplus_{i=1}^r \cO(a_i),
$$
where $\deg a_r \geq \cdots \geq \deg a_1$. 
\end{thm}

\begin{proof}

The proof proceeds by induction and is similar to that for $\PP^1$. Suppose the theorem is true for all vector bundles of rank $r-1$. Let $\cL_1 \subseteq \cV$ be the first step in a maximal filtration for $\cV$, and write $\cL_1 = \cO(-x^{\min})$, for some $x^{\min}\in \Pic(X)$. Twisting by $\cO(x^{\min})$ we obtain an exact sequence
$$
0\rightarrow \cO_{X} \rightarrow \cV(x^{\min}) \rightarrow \cF \rightarrow 0,
$$ 
for some vector bundle $\cF$ of rank $r-1$. By induction, we may write 
$$
\cF \cong \bigoplus_{i=1}^{r-1} \cO(b_i), \quad b_i \in \Pic(X).
$$
We now want to show that 
$$
\Ext^1(\cF, \cO) = H^1(X, \Hom(\cF,\cO)) = H^1(X, \bigoplus_{i=1}^{r-1} \cO(-b_i))
$$
vanishes, so that the above sequence is split and the theorem follows. Note that by Serre duality we have
$$
H^1(X, \bigoplus_{i=1}^{r-1} \cO(-b_i)) = H^0(X, \bigoplus_{i=1}^{r-1} \cO(b_i + \omega))
$$
where $\omega\in \Pic(X)$ is the dualizing element. Suppose first that there are two orbifold points $P_1,P_2$ of orders $(p_0,p_1)$. We have
$$
\Pic(X) = \{x_0,x_1 : p_0x_0 = p_1x_1\},
$$
so that $\omega = -x_0-x_1$. Consider then the twisted sequence
$$
0\rightarrow \cO(-x_0) \stackrel{s}\rightarrow \cV(x^{\min}-x_0) \rightarrow \cF(-x_0) \rightarrow 0,
$$ 
to which there is associated a long exact sequence in sheaf cohomology
$$
0\rightarrow H^0(X,\cO(-x_0)) \stackrel{s}\rightarrow H^0(X,\cV(x^{\min}-x_0)) \rightarrow H^0(X,\cF(-x_0)) \rightarrow H^1(X,\cO(-x_0))\rightarrow \ldots.
$$ 
Now
$$
\dim H^1(X,\cO(-x_0)) = \dim H^0(X,\cO(x_0 + \omega)) = \dim H^0(X,\cO(-x_0)) = 0,
$$
and $H^0(X,\cV(x^{\min}-x_0))=0$ by minimality of $x^{\min}$. Therefore
$$
H^0(X,\cF(-x_0)) = H^0(X, \bigoplus_{i=1}^{r-1} \cO(b_i - x_0)) = 0
$$
for all $i=1,\ldots, r-1$. But 
$$
\cO(b_i + \omega) = \cO(b_1 - x_0 - x_1) \subseteq \cO(b_i - x_0),
$$
and therefore $H^0(X,\cO(b_i + \omega))=0$ as well. If $X$ has only one orbifold point of order $p_0>1$, the Picard group is isomorphic to $\ZZ\,x_0$, $\omega = -c - x_0$ and the same argument applies. 
\end{proof}

\begin{rmk}
The decomposition in Theorem \ref{thm:SplittingTheorem} is unique whenever $\Pic_0(X) = 0$. In this case the degree map gives an isomorphism $\Pic(X)\cong \ZZ$ and the usual argument showing uniqueness of the decomposition for $\PP^1$ applies.  Note that by Prop. \ref{prop:picardGroupOfStackyLines} $\Pic_0(X)=0$  precisely when $X$ has at most one orbifold point, or two orbifold points of orders $p_0,p_1$ with $\gcd(p_0,p_1)=1$. 
\end{rmk}

When $X$ has three or more orbifold points, there exist vector bundles of rank two or higher that are indecomposable. According to \cite{Crawley-Boevey}, the enumeration of indecomposable vector bundles over an orbifold line $X$ of arbitrary signature can be done as follows. Suppose $X$ is an orbiline of signature $(p_0,\ldots, p_n)$ and consider the graph
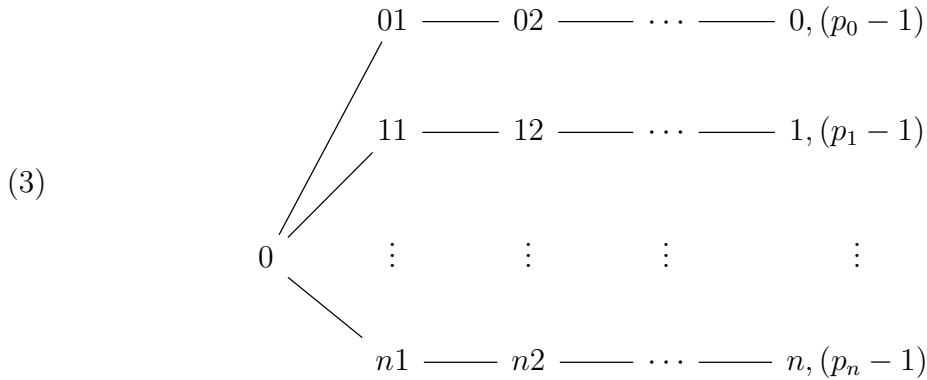
\begin{equation}
\label{graph:rootSystem}
\begin{tikzcd}
   & 01 \arrow[r,-]& 02 \arrow[r,-] & \cdots\arrow[r,-] & 0,(p_0-1) \\
      & 11 \arrow[r,-]& 12 \arrow[r,-] & \cdots\arrow[r,-] & 1,(p_1-1) \\
 0 \arrow[ur,-]\arrow[uur,-] \arrow[dr,-]& \vdots & \vdots & \vdots & \vdots \\
 &n1 \arrow[r, -]& n2 \arrow[r, -] & \cdots\arrow[r, -] & n,(p_n-1) 
\end{tikzcd}
\end{equation}
which is entirely determined by the signature of $X$.
Let $\Sigma$ be the $\ZZ$-module freely generated by the vertices of this graph, so that an element in $\Sigma$ can be written  as a linear combination $a_0\,0 + \sum_{i=0}^n \sum_{j=1}^{p_i-1} a_{ij} \,ij$. Let $\cV$ be a vector bundle over $X$. Then at each orbifold point $P_i$ the vector space $V_{i} \df\cV|_{P_i}$ has a filtration  
$$V_i \supseteq V_{i1} \supseteq \cdots \supseteq V_{i,p_i-1} $$
given by the eigenspace filtration of the isotropy representation $\mu(\cV,P_i)(1)$. The {\em dimension vector} of $\cV$ is the vector
$$
\underline{\rm{dim}}(\cV) \df \rk(\cV)\,0 + \sum_{i=0}^n \sum_{j=1}^{p_i-1}\dim(V_{ij}) \,ij \in \Sigma. 
$$
Let $\mathfrak{g}$ be the Kac-Moody algebra \cite{Kac} uniquely determined by the root system \eqref{graph:rootSystem}.

\begin{thm}[Crawley-Boevey, \cite{Crawley-Boevey}]
\label{thm:CB}
For each $d \in \ZZ$, there is an indecomposable vector bundle $\cV$ over $X$ of degree $d$ if and only if $\underline{\rm{dim}}(\cV)$ is a strict root for $\mathfrak{g}$. There is a unique such indecomposable vector bundle if $\underline{\rm{dim}}(\cV)$ is a real root, and infinitely may if $\underline{\rm{dim}}(\cV)$ is imaginary. 
\end{thm}

When the root system \eqref{graph:rootSystem} is a Dynkin diagram, the algebra $\mathfrak{g}$ is a simple Lie algebra and the classification of indecomposable vector bundles was already given in \cite{GeigleLenzing}, 5.4.1. In this case there are finitely many strict positive roots and thus bounds can be obtained on the rank of indecomposable vector bundles, by looking at the $0$-component of a maximal root. The only signatures for which $\mathfrak{g}$ is a simple Lie algebra are $(p_0)$, $(p_0,p_1)$, $(2,2,n),(2,3,3),(2,3,4)$ and $(2,3,5)$, corresponding to the Lie algebras $A_{p_0}$, $A_{p_0+p_1 -1}, D_{n+2}, E_6, E_7$ and $E_8$, respectively. We have already seen (Theorem \ref{thm:SplittingTheorem}) directly that for signatures $(p_0)$ and $(p_0,p_1)$ the maximal rank of an indecomposable vector bundle is one. We treat the case $(2,2,n)$ in detail below, where the bound on the maximal rank is 2. For $E_6, E_7$ and $E_8$ the maximal ranks for indecomposable vector bundles are $3$, $4$ and $6$, respectively (\cite{Lenzing}).

\begin{rmk}
Define the {\em virtual genus} of an orbifold curve to be 
$$
g_v(X)\df \frac{1}{2}\deg(\omega) + 1,
$$
in analogy with the usual relation between the genus and the degree of the canonical bundle for Riemann surfaces. The orbifold curves $X$ with $0 < g_v(X) < 1$ are precisely those whose root system \eqref{graph:rootSystem} is a Dynkin diagram, i.e. for which the rank of an indecomposable vector bundle is bounded. These orbifold curves seem to fit in between the two cases of Riemann surfaces of genus zero (i.e. $\PP^1$), where the maximal rank of indecomposable vector bundles is one, and of genus one (i.e. elliptic curves ), where there are indecomposable vector bundles of arbitrary rank (\cite{Atiyah}). 
\end{rmk}

\subsection{The case of signature $(2,2,n)$}

If $X$ has signature $(2,2,n)$ the Lie algebra $\mathfrak{g}$ of Theorem \ref{thm:CB} is the simple Lie algebra $D_{n+2}$ and the highest $0$-component of a strict root is 2. Thus for $X$ of signature $(2,2,n)$ there are indecomposable vector bundles of rank two, but of no higher rank. Moreover for a fixed degree $d$ there is a unique such rank two indecomposable vector bundle, by Thm. \ref{thm:CB}. We are able to describe this vector bundle explicitly:

\begin{thm}
\label{thm:ExistenceOfRank2Indec}
Let $X$ be the orbifold line of signature $(2,2,n)$. There is a unique (up to isomorphism) indecomposable vector bundle of rank two and degree -1/3 over $X$. This vector bundle is the unique non-trivial extension
$$
0 \rightarrow \Omega^1_X \rightarrow \cW \rightarrow \cO_X \rightarrow 0.
$$
\end{thm}

\begin{proof}
The extensions of the form
\begin{equation}
\label{equation:indecExtension}
0 \rightarrow \Omega^1_X \rightarrow \cW \rightarrow \cO_X \rightarrow 0
\end{equation}
are classified by the cohomology group
$$\mathrm{Ext}^1(\cO_X, \Omega^1_X) \cong H^1(X, \Hom(\cO_X, \Omega^1_X)) \cong H^1(X,\Omega^1_X).$$
Explicitly, recall that the element in $H^1(X, \Hom(\cO_X, \Omega^1_X))$ classifying $\cW$ is given by the connecting homomorphism $\delta$ in the long exact sequence of cohomology corresponding to the short exact sequence obtained by applying the functor $\Hom(\cO_X, -)$ to the exact sequence defining $\cW$. Since $\Hom(\cO_X, \cV) \cong \cV$ for any vector bundle $\cV$, the connecting homomorphism $\delta$ lies in the space
$$
\delta \in \Hom(H^0(X,\cO_X),H^1(X,\Omega^1_X)). 
$$
This vector space is 1-dimensional, thus there is a unique (up to isomorphism) extension $\cW$ giving $\delta\neq 0$ (i.e $\delta$ is an isomorphism). We claim that such $\cW$ is indecomposable. Indeed, tensoring by $\cW^*$ we get an exact sequence 
$$
0 \rightarrow \cW^*\otimes\Omega^1_X \rightarrow \cW^*\otimes\cW \rightarrow \cW^* \rightarrow 0,
$$
and thus a long exact sequence in cohomology 
$$
0 \rightarrow H^0(X,\cW^*\otimes\Omega^1_X) \rightarrow H^0(X,\cW^*\otimes\cW) \rightarrow H^0(X,\cW^*)  \rightarrow \ldots ,
$$
Now $\dim  H^0(X,\cW^*\otimes\Omega^1_X)  = \dim H^1(X,\cW) = 0$ since the long exact sequence coming from \eqref{equation:indecExtension} gives
$$
\ldots \rightarrow H^0(X,\cO_X) \stackrel{\delta}\rightarrow H^1(X,\Omega^1_X) \rightarrow H^1(X,\cW) \rightarrow H^1(X,\cO_X) = 0 \rightarrow 0
$$
and $\delta$ is an isomorphism by hypothesis. So we have an injection
$$
H^0(X,\cW^*\otimes\cW) = H^0(X,\End(\cW)) \hookrightarrow H^0(X,\cW^*).
$$
On the other hand, dualizing $\eqref{equation:indecExtension}$ and taking cohomology we get a long exact sequence
$$
0\rightarrow H^0(X,\cO_X) \rightarrow H^0(X, \cW^*) \rightarrow H^0(X,T_X) \rightarrow \ldots 
$$
where $T_X = (\Omega^1_X)^*$ is the tangent bundle of $X$. Now 
$$
\deg T_X = -1 + 1/2 + 1/2 + 1/n = 1/n >0
$$
thus $H^1(X,T_X) = 0$ and 
$$
\chi(T_X) = H^0(X,T_X) = 1/n + 1 - 1/2 - 1/2 - 1/n = 0.
$$

Therefore $\dim H^0(X, \cW^*) = \dim H^0(X,\cO_X) = 1$. This means that 
$$
\dim  H^0(X,\End(\cW)) \leq 1.
$$
If $\cW = \cO(a_1)\oplus \cO(a_2)$ were decomposable, then 
$$
\End(\cW) \cong \cO_X \oplus \cO_X \oplus \cO(a_2 - a_1) \oplus \cO(a_1 - a_2)
$$
and $\dim H^0(X,\End(\cW)) \geq H^0(X,\cO_X \oplus \cO_X ) = 2$, a contradiction. Thus $\cW$ is indecomposable.
\end{proof}

\begin{rmk}
It is easy to show that all the other indecomposable vector bundles on $X$ can be obtained from $\cW$ by twisting by a line bundle. 
\end{rmk}

\section{Coherent sheaves over orbifold lines}
\label{coherentSheaves}

Let $X$ be an orbifold line of signature $(p_0, \ldots, p_n)$, and for ease of notation let $W = \Pic(X)$, a finitely generated abelian group of rank one (Prop. \ref{prop:picardGroupOfStackyLines}). Let $\CC[W]$ be the group algebra of $W$ with coefficients in $\CC$, and let 
$$
G(W)\df \Spec(\CC[W])
$$
be the affine group scheme corresponding to it. The closed points of this group scheme can be identified with tuples $(t_0, \ldots, t_n) \in (\CC^{\times})^{n+1}$ such that $t_0^{p_0} = \cdots = t_n^{p_n}$. There is an action of $G(W)$ on $\mathbf{A}^{n+1}$ given on closed points by 
$$
(t_0, \ldots, t_n)(X_0, \ldots, X_n) = (t_0X_0, \ldots, t_nX_n).
$$

\begin{dfn}
The quotient stack $\PP(W) = \left[ G(W)\backslash \left(\mathbf{A}^{n+1} - \{0\}\right)\right]$ is called $W$-{\em weighted projective space}.
\end{dfn}

For each $i=0,\ldots,n$ consider the line bundle $\cO(x_i) = \cO_X(P_i)$, where $x_i\in W$ are the chosen  generators of $W$  as in Prop. \ref{prop:picardGroupOfStackyLines}. By the Riemann-Roch theorem, 
$$
\chi(\cO(x_i)) = \dim H^0(X,\cO(x_i)) - \dim H^1(X,\cO(x_i)) = \deg \cO(x_i) + 1 - \frac{1}{p_i} = 1,
$$
and thus we may choose a non-zero global section $s_i$ of $\cO(x_i)$ which vanishes only at the orbifold point $P_i$ and nowhere else. 

\begin{thm}
The tuple $(s_0, \ldots, s_n)$ gives a well-defined morphism 
$$
s: X \longrightarrow \PP(W)\\
$$
of $X$ into $W$-weighted projective space.
\end{thm}

\begin{proof}
First note that the $s_i$'s have no common zero since each $s_i$ vanishes precisely at the orbifold point $P_i$ and nowhere else. Therefore the collection $X_i = \{ x\in X: s_i(x) \neq 0\}$ is an open cover of $X$. Over each $X_i$, we may define a map 
$$
x\mapsto [s_0/s_i(x), \ldots, 1, \ldots, s_n/s_i(x)]
$$
to the standard open set $U_i = \{ X_i\neq 0\}$ of $\PP(W)$. This map is equivariant with respect to the action of the stabilizers of the orbifold points, hence it is well-defined. By descent, the local maps defined over each $X_i$ glue to give a global map $s: X\rightarrow \PP(W)$. 
\end{proof}

Next we show that $s$ is a closed immersion and we derive the equations defining $X$ as a subvariety in $\PP(W)$. Note that for each $i=0,\ldots, n$ the function $s_i^{p_i}$ is a section of $\cO(c)$. Riemann-Roch gives $\chi(\cO(c)) = 2$, but $\dim H^1(X,\cO(c))=0$ since the cohomology of $\cO(c)$ is the same as that of the usual twisting bundle $\cO(1)$ over the underlying Riemann surface $\PP^1$. Therefore the vector space of global section of $\cO(c)$ is 2-dimensional, spanned, say, by $s_0^{p_0}, s_1^{p_1}$ (these cannot be proportional since they have different zeroes $P_0\neq P_1$). It follows that for $i=2,\ldots, n$ we must have relations of the form
$$
s_i^{p_i} = \lambda_is_0^{p_0} -\mu_is_1^{p_1},\quad i=2,\ldots,n,
$$
for some $\lambda_i,\mu_i\in \CC$. By rescaling the $s_i$'s if necessary, we may assume that $\mu_i = 1$ for all $i=2,\ldots, n$. Let now $\CC[X_0, \ldots, X_n]$ be the coordinate ring of $\mathbf{A}^{n+1}$. The polynomials 
\begin{equation}
\label{equation:definingEquations}
X_i^{p_i} + X_1^{p_1} - \lambda_iX_0^{p_0}, \quad i=2,\ldots,n,
\end{equation}
cut out an affine subscheme $V$ of $\mathbf{A}^{n+1}$ together with a $G(W)$-action compatible with that on $\mathbf{A}^{n+1}$. We let
$$
C(W) \df \left[ G(W)\backslash \left(V-\{0\}\right)\right] \subseteq \PP(W)
$$
be the corresponding closed quotient sub-stack.

\begin{thm}
\label{thm:weightedProjectiveLineIsomorphism}
The morphism $s$ is a closed immersion inducing an isomorphism 
$$
s: X \stackrel{\cong}\longrightarrow C(W)\subseteq \PP(W)
$$
\end{thm}

\begin{proof}
The complex manifold underlying $\PP(W)$ is the usual projective space $\PP^n$, the map $\PP(W)\rightarrow \PP^n$ being given in coordinates by
$$
[X_0, \ldots, X_n] \longmapsto [X_0^{p_0}, \ldots, X_n^{p_n}].
$$
If we restrict this morphism to $C(W)$ we obtain the equations of a projective line in $\PP^n$. The map $C(W) \rightarrow \overline{C(W)}\cong\PP^1$ to the underlying Riemann surface of $C(W)$ is given in coordinates by 
$$
[z_0, \ldots, z_n] \longmapsto [z_0^{p_0},z_1^{p_1}].
$$
The map of underlying Riemann surfaces induced by $s$ is thus given by
\begin{align*}
x &\longmapsto [s_0^{p_0}(x),s_1^{p_1}(x)].
\end{align*}
Since $s_0^{p_0}, s_1^{p_1}$ generate the line bundle $\cO(1)$ over $\PP^1$, this map is clearly an isomorphism. Moreover $s$ is non-trivial at the orbifold points and it preserves the orbifold structure therefore it must be an isomorphism at the level of orbifolds as well. 
\end{proof}

\begin{rmk}
Once a choice of embedding $s:X\hookrightarrow \PP(W)$ as above is made, $X$ can be viewed as a {\em weighted projective line} \cite{GeigleLenzing}. The (weighted) projective coordinates of the orbifold points of order $p_0, p_1$ and $p_i$, $i = 2,\ldots, n$, of $X$ are precisely $\infty, 0$ and $\lambda_i$, $i = 2, \ldots, n$, respectively. 
\end{rmk}

Consider now the polynomial ring $\CC[X_0, \ldots, X_n]$ together with a $W$-grading given by $\deg \CC = 0$ and $\deg X_i = x_i$. Let $I(X)$ be the ideal of $\CC[X_0, \ldots, X_n]$ generated by the polynomials in \eqref{equation:definingEquations}. The quotient
$$
S(X)\df \CC[X_0, \ldots, X_n]/I(X) \cong \CC[z_0, \ldots, z_n]
$$
is again a $W$-graded ring generated by $n+1$ elements $z_i$ with $\deg z_i = x_i$. Let $\mathrm{grMod}(S(X))$ be the category of finitely generated graded modules over $S(X)$ (the morphisms being degree-preserving $S(X)$-module homomorphisms) and let $\mathrm{grMod}_0(S(X))$ be the full subcategory of modules that are finite dimensional vector spaces over $\CC$.

\begin{prop}
\label{prop:equivalenceCoherentSheaves}
The category $\mathrm{coh}(X)$ of coherent sheaves on $X$ is equivalent to the quotient category $\mathrm{grMod}(S(X))/\mathrm{grMod}_0(S(X))$.
\end{prop}

\begin{proof}
By Theorem \ref{thm:weightedProjectiveLineIsomorphism} $X$ is isomorphic to the quotient stack $C(W)=\left[ G(W)\backslash \left(V-\{0\}\right)\right]$, where $V$ is defined by \eqref{equation:definingEquations}. Thus  the category of coherent sheaves on $X$ is equivalent to that of $G(W)$-equivariant coherent sheaves on $V-\{0\}$. Now the category of coherent sheaves on  $V-\{0\}$ is the quotient category of the category of coherent sheaves on $V$ modulo the coherent sheaves with support at $0$, and the same is true for $G(W)$-equivariant sheaves. In turn, the category of $G(W)$-equivariant sheaves on $V$ is equivalent to $\mathrm{grMod}(S(X))$ via the sheafification functor
$$
\mathrm{grMod}(S(X)) \longrightarrow \mathrm{coh}([G(W)\backslash V]), \quad M \mapsto \widetilde{M},
$$
and the subcategory of those coherent sheaves with support on 0 is equivalent to $\mathrm{grMod}_0(S(X))$, since the closed point 0 corresponds to the ideal $(z_0, \ldots, z_n) \subseteq S(X)$. 
\end{proof}

The sheafification functor $
\mathrm{grMod}(S(X)) \longrightarrow \mathrm{coh}(X), M \mapsto \widetilde{M},
$ admits a right adjoint $\GM_*: \mathrm{coh}(X) \rightarrow \mathrm{grMod}(S(X))$ defined by
$$
\GM_*(\cF)\df \bigoplus_{x \in W} H^0(X, \cF(x)).
$$

\begin{rmk}
Traditionally (e.g. \cite{SerreFAC}, \cite{GeigleLenzing}) the functor $\GM_*$ is denoted by $\Gamma$. However, in what follows we would like to reserve the symbol $\Gamma$ to indicate a Fuchsian group. The letters $\GM$ are chosen to indicate that the $S(X)$-module obtained by applying the functor is of geometric origin.
\end{rmk}

We say that an $S(X)$-module $M$ is {\em maximal Cohen-Macaulay} if $M$ is (finitely generated) free over the sub-algebra $\CC[X_0^{p_0}, X_1^{p_1}]\hookrightarrow S(X)$.

\begin{thm}[Geigle-Lenzing, \cite{GeigleLenzing}, Thm. 5.1]
\label{t:geigleLenzing}
Suppose $\cV$ is a vector bundle on $X$. Then $\GM_*(\cV)$ is a maximal Cohen-Macaulay $S(X)$-module. 
\end{thm}

When $\cF = \cO(x)$ is a line bundle on $X$, the $S(X)$-module $\GM_*(\cO(x))$ is $S(X)[-x]$, the module obtained from $S(X)$ by shifting the $W$-grading by $x$. We say that a (finitely generated) $S(X)$-module is {\em free} if it is a direct sum of shifts of the form $S(X)[-x]$. In particular, Theorem \ref{thm:SplittingTheorem} gives: 

\begin{cor}
\label{cor:moduleSplitting}
Suppose $X$ has at most two orbifold points and let $\cV$ be a vector bundle of rank $r$ over $X$. Then the $S(X)$-module $\GM_*(\cV)$ is free of rank $r$. 
\end{cor}

\begin{rmk}
Alternatively, observe that if $X$ has at most two orbifold points then $S(X)$ is a polynomial ring in two variables, and any maximal Cohen-Macaulay $S(X)$-module over a polynomial ring is free. 
\end{rmk}

Conversely:

\begin{prop}
\label{p:notfree}
Suppose $\cV$ is an indecomposable vector bundle over the orbifold line $X$. Then $\GM_*(\cV)$ is not free. 
\end{prop}

\begin{proof}
Suppose $\GM_*(\cV) \cong \oplus_{i=1}^r S[-a_i]$ for some $a_i\in W = \Pic(X)$. The sheafification functor gives $\widetilde{\GM_*(\cV)} \cong \oplus_{i=1}^r \cO(a_i)$. But $\GM_*$ is right-adjoint to $\cF$, so there is a canonical isomorphism $\widetilde{\GM_*(\cV)}\cong \cV$, contradicting the fact that $\cV$ is indecomposable. 
\end{proof}

For example, let $X$ be the orbifold line of signature $(2,2,2)$ and let $\cV$ be the rank two indecomposable vector bundle $\cW$ of Theorem \ref{thm:ExistenceOfRank2Indec}. Then $\GM_*(\cW)$ is not free over $S(X)$.

\section{Vector valued modular forms}
\label{s:vvmfs}

Let $\Gamma\subseteq\PSL_2(\RR)$ denote a Fuchsian group of the first kind, and assume that $\Gamma$ is of finite covolume. Let $\tau_0, \ldots, \tau_n$ denote the elliptic points of $\Gamma$, each of order $p_0, \ldots, p_n$, respectively. Let $s_1,\ldots, s_m$ denote the cusps of $\Gamma$. The group $\Gamma$ acts on $\mathfrak{h}$ by linear fractional transformations and the orbifold quotient $\left[ \Gamma\backslash\mathfrak{h} \right]$ is a one-dimensional complex orbifold which is a Riemann surface away from the elliptic points and it has a $\ZZ/p_i\ZZ$-orbifold structure around each elliptic point $\tau_i$. 

Let $\sM_{\Gamma}$ denote the compact {\em modular orbifold} associated to $\Gamma$, the orbifold curve obtained by compactifiying the quotient $\left[\Gamma\backslash\mathfrak{h} \right]$ by adding the cusps of $\Gamma$ (if any). This is obtained from $\left[\Gamma\backslash\mathfrak{h} \right]$ by glueing disks around each cusp of $\Gamma$, as follows. For $s$ a cusp, let $\alpha \in \PSL_2(\RR)$ be such that $\alpha(s) = \infty$. The stabilizer $\Gamma_s$ of the cusp $s$ satisfies
$$
\alpha \Gamma_s \alpha^{-1} = \left\{ \twomat 1h01^m: m\in \ZZ \right\}\cong \ZZ, 
$$
for some real number $h > 0$, called the {\em width} of $s$. There is a canonical injection $\iota_s: \left[ \Gamma_s \backslash \mathfrak{h} \right]\hookrightarrow \left[\Gamma\backslash\mathfrak{h} \right]$ and $\left[ \Gamma_s \backslash \mathfrak{h} \right]$ maps holomorphically to the punctured unit disk $\DD^{\times}$ via the map
$$
\tau \longmapsto e^{2\pi i \alpha(\tau)/h}.
$$
The punctured disk admits a canonical compactification $\iota: \DD^{\times} \hookrightarrow \DD$ by the unit disk, so the compactification along $s$ is obtained via the following diagram 

\begin{equation}
\label{diagram:compactification}
\xymatrixcolsep{5pc}\xymatrix{
&\left[ \langle \Gamma_s \rangle\backslash\frakh\right]\ar[ld]_{\iota_s}\ar@{<->}[r]^{\tau \mapsto e^{2\pi i \alpha(\tau)/h} } & \DD^\times \ar[rd]^{\iota} &\\
\left[\Gamma\backslash\frakh\right] = \sM_{\Gamma} &&& \DD ,}
\end{equation}
which identifies the cusp $s$ with the origin $0 \in \DD$.

The group $\Gamma$ is identified with the orbifold fundamental group of $[\Gamma \backslash \uhp]$. The following result is well-known.
\begin{prop}
\label{p:fuchsianpresentation}
Assume that the compact curve $\sM_{\Gamma}$ is of genus zero. Then the Fuchsian group $\Gamma$ admits a presentation of the form
\[
  \Gamma =\langle e_0,\ldots, e_n, \sigma_1\ldots, \sigma_m\mid e_j^{p_j}=1 \textrm{ for all }j,~ e_0\cdots e_n\sigma_1\cdots \sigma_m = 1 \rangle.
\]
The $e_j$ may be taken to be generators of the elliptic isotropy subgroups, and the $\sigma_j$ may be taken to be generators of the cuspidal isotropy subgroups. In particular, if $\Gamma$ has a unique cusp, then $\Gamma$ is the free product of its elliptic isotropy subgroups.
\end{prop}
\begin{proof}
This is a classical result that goes back to Klein and Poincar\'{e}, but it is hard to find a good reference. See the book \cite{Stillwell} for a discussion that treats the general case of higher genus orbifolds.
\end{proof}

Let $M_{\Gamma}$ be the Riemann surface underlying $\sM_{\Gamma}$, a compact, connected Riemann surface equipped with a map $j:\sM_{\Gamma}\rightarrow M_{\Gamma}$ which is universal among all maps from $\sM_{\Gamma}$ to a Riemann surface. Assume that the genus of $\Gamma$ is zero, so that $\sM_{\Gamma}$ is an orbifold line and there is an isomorphism of Riemann surfaces
$$
M_{\Gamma} \stackrel{\cong}\longrightarrow \PP^1. 
$$

By general descent theory, a vector bundle on $\sM_{\Gamma}$ can be specified by giving a vector bundle $\cV^{\circ}$ on the open quotient stack $\left[\Gamma\backslash\mathfrak{h}\right]$, together with the data of a pair $(\cW_s, \phi_s)$, for each cusp $s$, of a vector bundle $\cW_s$ over $\DD$ and an isomorphism of vector bundles
$$
\iota_s^* \cV^{\circ} \cong \iota^*\cW_s
$$
lying over the isomorphism $\tau\mapsto e^{2\pi i \alpha(\tau)/h}$ (notation as in \eqref{diagram:compactification}). In particular a complex, finite dimensional representation 
$$
\rho: \Gamma \longmapsto \GL(V),
$$ 
gives by definition a local system $\cV^{\circ}(\rho)$ on $\left[\Gamma\backslash\mathfrak{h}\right]$. This local system maybe extended to a vector bundle over all of $\sM_{\Gamma}$ as follows: for each cusp $s$, let $T_s \df \stwomat 1h01$ and let $\gamma_s \df \alpha^{-1}T_s\alpha$. Choose $L_s \in \End(V)$ such that 
$$
e^{2\pi i L_s} = \rho(\gamma_s).
$$
The extension of $\cV^{\circ}(\rho)$ to $\sM_{\Gamma}$ at each cusp $s$ is then given by the trivial vector bundle $\cO_{\DD}^{\oplus \dim \rho}$ over $\DD$ together with the isomorphism
$$
\phi_s(v,\tau) = (e^{-2\pi i \alpha(\tau) L_s/h}\,v, e^{2 \pi i \alpha(\tau)/h} )
$$
of $\iota_s^*\cV^{\circ}(\rho)$ with $\iota^*\cO_{\DD}^{\oplus \dim \rho}$. 

\begin{dfn}
Let $L = \{ L_s \}_{s\in \mathrm{cusps}(\Gamma)}$ be the vector of matrices chosen as above. The extension of the vector bundle $\cV^{\circ}(\rho)$ to $\sM_{\Gamma}$ given above is called the {\em extension corresponding to the choice of exponents $L$} and it is denoted by $\cV_L(\rho)$. If all eigenvalues of each $L_s$ have real part contained in $[0,1)$, we say that $L$ is the {\em standard choice of exponents} and that $\cV_L(\rho)$ is {\em the canonical extension} of  $\cV^{\circ}(\rho)$, denoted by $\cV(\rho)$. If all eigenvalues of each $L_s$ have real part contained in $(0,1]$, we say that $L$ is the {\em cuspidal choice of exponents} and that $\cV_L(\rho)$ is {\em cuspidal extension} of  $\cV^{\circ}(\rho)$, denoted by $\cS(\rho)$.
\end{dfn}

There is a special line bundle $\cL_2$ over $\sM_{\Gamma}$ of {\em modular forms of weight two}, which is the extension of the line bundle given on the open quotient $\left[\Gamma\backslash\mathfrak{h}\right]$ by the 1-cocycle
$$
\twomat abcd \longmapsto (c\tau + d)^2,
$$
and characterized by the isomorphism
$
\cL_2 \cong \Omega^1_{\sM_{\Gamma}}\left(\log \sum s_i\right),
$
so that
$$
\deg \cL_2 = \deg \omega + \#\,\mathrm{cusps}(\Gamma) = (n-1) - \sum_{i=0}^n \frac{1}{p_i} + \#\,\mathrm{cusps}(\Gamma). 
$$

For any $k\in 2\ZZ$, the global sections of $\cL_k\df\cL_2^{\otimes k/2}$ are precisely the classical {\em holomorphic modular forms of weight $k$} on the group $\Gamma$, whose space is denoted by $M_k(\Gamma)$. Let 
$$
R(\Gamma)\df \bigoplus_{k\in 2\ZZ} M_k(\Gamma)
$$
be the ring of modular forms over $\Gamma$. If $\cV(\rho)$ is the canonical extension of a local system given by $\rho$, let 
$$
\cV_k(\rho)\df \cV_k(\rho)\otimes\cL_k.
$$

\begin{dfn}
A {\em holomorphic, $\rho$-valued modular form of weight $k$ on $\Gamma$} is a global section of $\cV_k(\rho)$ over $\sM_{\Gamma}$. The space of all such sections is denoted by $M_k(\rho)$. The  {\em module of $\rho$-valued modular forms} is the $\ZZ$-graded $R(\Gamma)$-module
$$
M(\rho)\df \bigoplus_{k\in 2\ZZ} M_k(\rho).
$$
\end{dfn}

The rings $R(\Gamma)$ are classical and well-studied objects (see \cite{VoightZB} and the references therein). The structure of the $R(\Gamma)$-modules $M(\rho)$ are however unknown except in level one, where it is known that $M(\rho)$ is always free of rank $\dim \rho$ \cite{MarksMason}. The goal of the following sections is to demonstrate that the modules $M(\rho)$ are in fact poorly behaved, and that a more well-behaved object $\GM(\rho)$ is given by a {\em geometrically weighted} version of $M(\rho)$. In order to define it, let 
$$
S(\Gamma)\df S(\sM_{\Gamma})
$$
denote the ring obtained by viewing $\sM_{\Gamma}$ as an orbifold line canonically embedded into weighted projective space as in Section \ref{coherentSheaves}. This is a $W(\Gamma)$-graded ring, where $W(\Gamma)\df \Pic(\sM_{\Gamma})$ is a finitely generated abelian group of rank one (Proposition \ref{prop:picardGroupOfStackyLines}). It is the ring generated by the spaces of global sections of \emph{all} of the line bundles on $\sM_\Gamma$, not just the modules $\cL_k$.

\begin{dfn}
The module of {\em geometrically weighted} $\rho$-valued modular forms  is the $W(\Gamma)$-graded $S(\Gamma)$-module
$$
\GM(\rho)\df \GM_*(\cV(\rho)).
$$
\end{dfn}

The modules $GM(\rho)$ of geometrically weighted modular forms seem to be easier to study than the modules $M(\rho)$, since their structure mirrors that of vector bundles over $\sM_{\Gamma}$. For example, the following general result is a special case of Corollary \ref{cor:moduleSplitting}:

\begin{thm}
\label{thm:FMT}
Suppose that $\Gamma$ has at most two elliptic points, and let $\rho: \Gamma \longmapsto \GL(V)$ be any complex finite dimensional representation.  Then $\GM(\rho)$ is free of rank $\dim \rho$ over $S(\Gamma)$.
\end{thm}

One would like to have a purely automorphic description of the ring $S(\Gamma)$ and the modules $\GM(\rho)$. Since the upper half plane is a Stein manifold, all vector bundles on curves uniformized by Fuchsian groups pull back to bundles described by automorphy factors. Thus, an automorphic description of these objects exists, however it seems to be a slightly delicate problem to obtain an explicit one in generality. In certain cases, however, one has access to a convenient description. To state it, we first give an automorphic description of $\Pic_0(\sM_\Gamma)$.
\begin{dfn}
A \emph{cuspidal character} of a Fuchsian group $\Gamma$ is a rank one representation $\chi\colon \Gamma\to \CC^\times$ such that $\chi(\sigma_j) = 1$ for each cuspidal isotropy generator $\sigma_j$ in Proposition \ref{p:fuchsianpresentation}.
\end{dfn}
Since the fundamental group of $\sM_\Gamma$ is the quotient of the fundamental group of $[\Gamma \backslash \uhp]$ obtained by throwing away the cuspidal generators, and since $\Pic_0(\sM_\Gamma) = \Hom(\pi_1(\sM_\Gamma),\CC^\times)$, we find that the group of cuspidal characters of the Fuchsian group $\Gamma$ is naturally identified with $\Pic_0(\sM_\Gamma)$. Thus from the exact sequence
\[
  0 \to \Pic_0(\sM_\Gamma) \to \Pic(\sM_\Gamma) \to \frac{1}{m}\ZZ \to 0,
\]
we see that to give an automorphic description of $S(\Gamma)$, it suffices to give an automorphic description of a line bundle of degree $1/m$, where $m = \lcm(p_0,\ldots, p_n)$. The most convenient cases are when $\cL_2$ is of degree $\frac 1m$. This holds in the cases of signature $(2,3)$, $(3,3)$ and $(2,2,2)$ discussed below, but not in the last case of signature $(2,2,3)$ that we consider.

When $\cL_2$ is of degree $\frac 1m$, one finds that
\begin{align*}
  S(\Gamma) &\cong \bigoplus_{(\chi,k)} M_k(\Gamma,\chi), & \GM(\rho) & \cong \bigoplus_{(\chi,k)} M_k(\Gamma, \rho\otimes \chi),
\end{align*}
where the direct sums are over all of the cuspidal characters $\chi$ and integer weights $k$. Crucial to this description is the fact that if $L$ denotes a choice of exponents for a representation $\rho$, then one can use the exact same exponents for $\rho \otimes \chi$ for any cuspidal character $\chi$. Hence one has $\cV_{k,L}(\rho\otimes \chi) = \cV_{0,L}(\rho)\otimes \cV_{k}(\chi)$, where $\cV_{k}(\chi)$ denotes the canonical extension of the cuspidal character $\chi$. For characters that are not cuspidal, one need not have compatibility of extensions like this (instead one must adjust the exponents on the left side of the identity). Note also that since this holds for an arbitrary choice of exponents for $\rho$, one can define $\GM_L(\rho)$ using these exponents, and one obtains an analogous description in terms of the $M_{k,L}(\Gamma,\rho\otimes \chi)$. However, it is not usually the case that $\cL_2$ is of degree $\frac 1m$, and so in most cases the ring $\bigoplus_{(\chi,k)} M_k(\Gamma,\chi)$ generated by the cuspidal characters is strictly smaller than $S(\Gamma)$. Our final example below illustrates this phenomenon.

\section{Examples and counterexamples}
\label{s:examples}

This section presents several examples that illustrate the theory above.

\subsection{Level one}
Let $\Gamma(1) = \PSL_2(\ZZ)$. This group has two elliptic points $\tau_0 = i$ and $\tau_1 = e^{2\pi i/3}$ of order 2 and 3, respectively. The abelian group $W=\Pic(\sM_{\Gamma(1)})$ is generated by two elements $x_0$ and $x_1$ such that $2x_0 = 3x_1$, and $\PP(W)$ is the weighted projective line $\PP(2,3)$ (as in \cite{CandeloriFranc}). There is a unique order-preserving isomorphism $W\cong \ZZ$ such that $x_0$ and $x_1$ correspond to the elements $3$ and $2$, respectively. The canonical generator for $W$ is the class of the line bundle of weight $2$ modular forms $\cL_2$, of degree $\frac 16$. The line bundle $\cO(x_0)$, of degree $\frac{1}{2}$, must then be isomorphic to $\cL_6$ and by Riemann-Roch it has a unique (up to rescaling) global section $E_6$ with a simple zero at the elliptic point of order $2$. Similarly, the line bundle $\cO(x_1)$ is isomorphic to $\cL_4$ and it has a unique (up to rescaling) global section $E_4$ with a simple zero at the elliptic point of order 3. The isomorphism 
$$
s:\sM_{\Gamma(1)} \stackrel{\cong}\rightarrow C(W) 
$$
in this case is given in coordinates by $\tau \mapsto [E_6(\tau), E_4(\tau)]$. The map induced by $s$ at the level of underlying Riemann surfaces $M_{\Gamma(1)}\rightarrow \PP^1$ is given in coordinates by 
$$
\tau \mapsto [E^2_6(\tau), E^3_4(\tau)],
$$
which sends $\tau_0 = i$ to $\lambda_0 = \infty = [0,1]$ and $\tau_1 = e^{2\pi i/3}$ to $\lambda_1 = 0 = [1,0]$. Up to a linear change of coordinates in $\PP^1$, this map  corresponds with the usual $j$-function. Note that since $W \cong \ZZ$, generated by $\cL_2$, we have 
$$
S(\Gamma(1)) = R(\Gamma(1)) = \CC[E_4,E_6]
$$
and similarly $M(\rho) = \GM(\rho)$, for any representation $\rho: \Gamma(1) \rightarrow \GL(V)$. Thus, for $\Gamma(1)$, geometrically weighted modules of modular forms correspond to the usual $\ZZ$-graded modules of modular forms. In particular, Theorem \ref{thm:FMT} specializes in level one to the well-known {\em free module theorem} for vector valued modular forms on $\Gamma(1)$ (\cite{MarksMason}, \cite{Gannon}, \cite{CandeloriFranc}).

\subsection{The subgroup of index two}
\label{ss:indexTwo}
Let $\Gamma^2$ denote the subgroup of $\PSL_2(\ZZ)$ of index two. This group is discussed in Section 1.3 of \cite{Rankin}. If $\chi$ denotes the generating character of $\PSL_2(\ZZ)$ satisfying $\chi(T) = e^{\frac{2\pi i}{6}}$, then $\Gamma^2 = \ker \chi^3$, so that $T$ is not an element of $\Gamma^2$. Similarly $S$ is not an element of $\Gamma^2$, while $R$ is an element of $\Gamma^2$. Thus, $\tau_0 = \zeta = e^{2\pi i/3}$ and $\tau_1 = \zeta+1$ are two elliptic points for $\Gamma^2$ with stabilizers  $R_0 =R$ and $R_1 = TRT^{-1}$, respectively, and there are no other elliptic points. Since $\Gamma^2$ has a unique cusp, it follows by Proposition \ref{p:fuchsianpresentation} that $\Gamma^2$ is freely generated by $R_0$ and $R_1$. Since $R_0$ and $R_1$ are of order three, it follows that $\chi^3$ restricts to the trivial character of $\Gamma^2$, but that $\chi$ and $\chi^2$ are nontrivial. Observe that $R_1R_0 = T^2$ in $\PSL_2(\RR)$.

Define characters $\alpha_0,\alpha_1 \in\Hom(\Gamma^2,\CC^\times)$ by setting $\alpha_i(R_i) = \zeta$ and $\alpha_i(R_j) = 1$ if $i \neq j$. Since $\chi(R_0) = \chi(R_1) = \zeta^2$, we have $\chi = \alpha_0^2\alpha_1^2$. The subgroup of cuspidal characters is $\{1,\alpha_0\alpha_1^2,\alpha_0^2\alpha_1)$. If $\phi \in \Hom(\Gamma^2,\CC^\times)$, then let $\phi^T$ denote the conjugate character defined by
\[
  \phi^T(g) = \phi(T^{-1}gT).
\]
If $\rho = \Ind\phi$, where by $\Ind$ we mean the induction up to $\PSL_2(\ZZ)$, then observe that there exists a basis for $\rho$ such that
\begin{align*}
\rho(T) &= \twomat 0{\phi(T^2)}10, & \rho(S) &= \twomat 0{\phi(R)}{\phi(T^{-1}RT^{-1})}0, & \rho(R) &= \twomat{\phi(R)}{0}{0}{\phi^T(R)}.
\end{align*}
In particular, $(\det\Ind\phi)(T) = -\phi(T^2)$. Mackey's criterion for irreducibility implies that $\Ind \phi$ is irreducible if and only if $\phi\neq \phi^T$. The characters fixed by conjugation are precisely the characters $\chi^r = \alpha_0^r\alpha_1^{r}$, and for each of these one has $\Ind \chi^r \cong \chi^r\oplus \chi^{r+3}$. 

To determine the pairs $(\phi,\phi^T)$ for the remaining characters, observe that since $R_0 = ST$ and $R_1 = TS$, one has $\alpha_0^T = \alpha_1$. Thus, the characters $\alpha_0^u\alpha_1^u$ have reducible inductions, while the characters in each pair $(\alpha_0^u\alpha_1^v,\alpha_0^v\alpha_1^u)$ for $u \neq v$ yield irreducible and isomorphic inductions. The following table lists the exponents for the induced representations and the minimal weight, obtained via the isomorphism $M(\phi) \cong M(\Ind \phi)$ and results about vector valued modular forms of rank two for $\SL_2(\ZZ)$ (cf. \cite{FrancMason1} or \cite{FrancMason2}). 
\[
\def\arraystretch{1.5}
\begin{array}{c|c|c|c}
(\phi,\phi^T) &(\alpha_0\alpha_1^2,\alpha_0^2\alpha_1)&  (\alpha_0,\alpha_1)& (\alpha_0^2,\alpha_1^2)\\
\hline
\phi(T^2) & 1& \zeta & \zeta^2\\
\hline
\textrm{Exponents for } \Ind \phi & 0, \frac 12&\frac 16, \frac 23 &\frac{1}{3},\frac{5}{6}\\
\hline
\textrm{Minimal weight} &2&4&6
\end{array}
\]
Moving to the right along this table corresponds to tensoring with $\chi$. In general, multiplication by $\eta^4$ defines an inclusion $M(\rho) \hookrightarrow M(\rho\otimes \chi)$. For $\rho$ equal to a character $\phi$ as in the table, one sees that this map is in fact an isomorphism, since the minimal weights increase by two with each twist.

For ease of notation set $\beta = \alpha_0\alpha_1^2$. In the summer of 2016, Geoff Mason showed that not all modules $M(\rho)$ of vector valued modular forms for $\Gamma^2$ are free over the ring $R$ of scalar valued modular forms for $\Gamma^2$. Since $R$ is not a polynomial ring, it is natural to ask whether the modules $M(\rho)$ are at least projective over $R$, in analogy with the free module theorem in level one. We will show that $M(\beta)$ is not a projective module over the ring 
\[R = \frac{\CC[E_4,E_6,\eta^{12}]}{(E_4^3-E_6^2-12^3\eta^{24})}\] 
of scalar valued modular forms for $\Gamma^2$. Here $E_k$ denotes the Eisenstein series of weight $k$ for $\Gamma(1)$, normalized so that its constant term is $1$, and $\eta$ denotes the Dedekind eta function.

It is easy to see that $M(\beta)$ is not a \emph{free} $R$-module. This is because $\Ind(\beta)$ is irreducible, and since $M(\beta) \cong M(\Ind \beta)$, it follows that the Hilbert-Poincar\'{e} series $\sum_{k \in \ZZ} \dim M_k(\beta)T^k$ for the graded module $M(\beta)$ is equal to $\frac{T^2+T^4}{(1-T^4)(1-T^6)}$ (cf. \cite{FrancMason1} or \cite{FrancMason2}). However, a free graded $R$-module has a Hilbert-Poincar\'{e} series that is a sum of series of the form $\frac{T^{l}+T^{l+6}}{(1-T^4)(1-T^6)}$ for various weights $l$. It follows that $M(\beta)$ is not a free $R$-module.

The exponents of $\rho = \Ind\beta$ are $0$ and $1/2$. This corresponds to Example 21 in \cite{FrancMason2}. There exists a basis for $M(\beta)$ in terms of the theta series\footnote{Here $q = e^{2\pi i \tau}$ rather than $q = e^{\pi i \tau}$.}
\begin{align*}
\theta_2(q) &= 2\sum_{n = 0}^{\infty} q^{\frac{(2n+1)^2}{8}}, & \theta_3(q) &= 1+2\sum_{n\geq 1} q^{n^2/2}, & \theta_4(q) &= 1+2\sum_{n\geq 1}(-1)^nq^{n^2/2}.
\end{align*}
Using the transformation laws of theta series, and the classical fact that $\theta_3^4=\theta_2^4+\theta_4^4$, one sees that
\[
  F = (1 + e^{2\pi i \frac 16})\theta_2^4-e^{2\pi i \frac 56}(\theta_3^4+\theta_4^4)
\]
is the unique, up to rescaling, form $F \in M_2(\beta)$. Further, the structure theory of vector valued modular forms of rank two (cf. \cite{Mason}, \cite{FrancMason1} or \cite{FrancMason2}) implies that $F$ and $DF$ generate $M(\beta)$ freely as an $R(1) = \CC[E_4,E_6]$ module. By examining $q$-expansions, one finds that
\begin{align*}
\eta^{12}F &= u\left(E_6F+6E_4DF\right), & \eta^{12}DF &= -\frac{1}{6}u\left(E_4^2F+6E_6DF\right),
\end{align*}
where $u = \frac{1}{72}(2e^{2\pi i\frac 16}-1)$ is a square root of $-12^{-3}$.

\begin{prop}
\label{p:projectiveFail}
The module $M(\beta)$ of vector valued modular forms for $\beta$ is not projective over the ring $R = \CC[E_4,E_6,\eta^{12}]$ of scalar valued modular forms for $\Gamma^2$.
\end{prop}
\begin{proof}
Consider the exact sequence
\begin{equation}
\label{eq:exseq}
  0 \to \ker p \to R\oplus R \stackrel{p}{\to} M(\beta) \to 0
\end{equation}
of $R$-modules defined by $p(x,y) = xF+yDF$. We will show that (\ref{eq:exseq}) is not split, and hence $M(\beta)$ is not a projective $R$-module.

Let $(a+b\eta^{12},c+d\eta^{12}) \in \ker p$ where $a$, $b$, $c$ and $d$ are classical scalar forms of level one. By hypothesis
\begin{align*}
  0 &= aF+b\eta^{12}F +cDF + d\eta^{12}DF\\
&= aF+bu(E_6F+6E_4DF)+cDF-\frac{1}{6}du\left(E_4^2F+6E_6DF\right)\\
&= \left(a+buE_6-\frac{1}{6}duE_4^2\right)F+\left(c+6buE_4-duE_6\right)DF 
\end{align*}
Thus, after relabeling variables,
\[
  \ker p = \left\{\left(u(bE_4^2-aE_6)+a\eta^{12}, 6u(bE_6-aE_4)+6b\eta^{12}\right) \mid a,b \in R(1)\right\}.
\]
An $R(1)$-basis for $\ker p$ is given by
\begin{align*}
  e_1 &= (-uE_6+\eta^{12},-6uE_4), & e_2 &= (uE_4^2,6uE_6+6\eta^{12}).
\end{align*}
Observe that since $-12^{3}u^2 = 1$, then $\Delta = u^2(E_6^2-E_4^3)$. It follows that the matrix of multiplication by $\eta^{12}$ acting on $\ker p$ in this $R(1)$-basis is
\[u\cdot \twomat {-E_6}{E_4^2}{-E_4}{E_6}.\]

The preimages of $F$ and $DF$ under $p$ must be of the form
\begin{align*}
  b_1 &= (1,0) + ae_1+be_2,&  b_2 &= (0,1) + ce_1+de_2,
\end{align*}
for unique scalar forms $a,b,c,d \in R(1)$. In order for the exact sequence (\ref{eq:exseq}) to split, it must be possible to find choices of $b_1$ and $b_2$ such that $R(1)b_1\oplus R(1)b_2$ is stable under multiplication by $\eta^{12}$. To see that no such choice exists, observe that
\begin{align*}
  \eta^{12}b_1 &=(\eta^{12},0) + u(bE_4^2-aE_6)e_1+u(bE_6-aE_4)e_2,\\
  \eta^{12}b_2 &=(0,\eta^{12}) + u(dE_4^2-cE_6)e_1+u(dE_6-cE_4)e_2.
\end{align*}
Solving $xb_1+yb_2 = \eta^{12}b_1$ for $x,y\in R(1)$ amounts to solving
\begin{align*}
\eta^{12}b_1=  &(x,y) + (xa+yc)e_1+(xb+yd)e_2\\
  =&(x,y) -u(xa+yc)(E_6,6E_4)+u(xb+yd)(E_4^2,6E_6)\\
   &\quad  + (xa+yc)(\eta^{12},0)+(xb+yd)(0,6\eta^{12})
\end{align*}
On the other hand,
\begin{align*}
  \eta^{12}b_1 =& (1+u(bE_4^2-aE_6))(\eta^{12},0)+u(bE_4^2-aE_6)(-uE_6,-6uE_4)+\\
  &\quad u(bE_6-aE_4)(0,6\eta^{12})+u(bE_6-aE_4)(uE_4^2,6uE_6)
\end{align*}
These considerations lead to the equations
\begin{align*}
  xa+yc &= 1+u(bE_4^2-aE_6),\\
  xb+yd &= u(bE_6-aE_4),\\
  x-u(xa+yc)E_6+u(xb+yd)E_4^2 &= -u^2(bE_4^2-aE_6)E_6+u^2(bE_6-aE_4)E_4^2,\\
  y-6u(xa+yc)E_4+6u(xb+yd)E_6 &= -6u^2(bE_4^2-aE_6)E_4+6u^2(bE_6-aE_4)E_6.
\end{align*}
Substituting the first two equations above into the last two yields $x = uE_6$ and $y = 6uE_4$. But then the first equation is not satisfied. This shows that the exact sequence (\ref{eq:exseq}) is not split as an $R$-module, and hence $M(\beta)$ is not a projective $R$-module.
\end{proof}

In this case since $1$, $\beta$ and $\beta^2$ are the cuspidal characters, and since $\cL_2$ is of degree $1/3$, one has
\[
  S(\Gamma^2) = M(\Gamma^2,1)\oplus M(\Gamma^2,\beta)\oplus M(\Gamma^2,\beta^2).
\]
The minimal weights for $\beta$ and $\beta^2$ are both $2$, and $\dim M_2(\Gamma^2,\beta) = \dim M_2(\Gamma^2,\beta^2) = 1$. If $F$ and $G$ denote generators for these spaces, then $S(\Gamma^2) = \CC[F,G]$. If $\rho$ is any representations of $\Gamma^2$, and if $L$ denotes any choice of exponents for $\rho(T^2)$, then Corollary \ref{cor:moduleSplitting} states that $\GM_L(\rho) = M_L(\rho) \oplus M_L(\rho \otimes \beta) \oplus M_L(\rho \otimes \beta^2)$ is a free module over $\CC[F,G]$ of rank $\dim \rho$. This is the appropriate generalization of the free module theorem of \cite{MarksMason} to the subgroup of $\PSL_2(\ZZ)$ of index $2$. 

\begin{rmk}
The free module theorem for $\Gamma^2$ has been obtained independently by Gannon using a different argument (unpublished note). Gannon has also proved similar results for other subgroups, but it is not clear to the authors whether Gannon's method generalizes as easily as the geometric arguments presented here.
\end{rmk}

\subsection{The normal subgroup of index three}
\label{example:normalIndex3}
Let $\Gamma^3$ denote the unique normal subgroup of $\PSL_2(\ZZ)$ of index $3$, which is discussed in Section 1.2 of \cite{Rankin}. Since $\Gamma^3 = \ker \chi^2$, it follows that $T^3 \in \Gamma^3$, $S \in \Gamma^3$ and $R \not \in \Gamma^3$. The elliptic points are $\tau_0 = i$, $\tau_1 = i+1$ and $\tau_2 = i+2$. The matrices $S_0 = S$, $S_1 = TST^{-1}$, $S_2 = T^2ST^{-2}$ generate the isotropy subgroups, and thus by Proposition \ref{p:fuchsianpresentation}, they generate $\Gamma^3$ freely. 

Define characters of $\Gamma^3$ by setting $\alpha_i(S_i) = -1$ and $\alpha_i(S_j) = 1$ for $j \neq i$. Then $\alpha_0$, $\alpha_1$ and $\alpha_2$ generate the character group freely. Let $\chi$ denote the restriction of the character of $\eta^4$ for $\PSL_2(\ZZ)$ to $\Gamma^3$. Then $\chi(S_i) = -1$ for all $i$, so that $\chi = \alpha_0\alpha_1\alpha_2$. One easily checks that conjugation by $T$ permutes the generating characters by $\alpha_0 \mapsto \alpha_1 \mapsto \alpha_2 \mapsto \alpha_0$. The orbit decomposition for the characters is thus
\[
  \Hom(\Gamma^3,\CC^\times) = \{1\} \cup \{\alpha_0\alpha_1\alpha_2\} \cup \{\alpha_0,\alpha_1,\alpha_2\} \cup \{\alpha_1\alpha_2,\alpha_0\alpha_2,\alpha_0\alpha_1\}.
\]
The cuspidal characters are $\{1,\alpha_1\alpha_2,\alpha_0\alpha_2,\alpha_0\alpha_1\}$.

The following table lists the exponents and minimal weights for $\Ind \phi$ as $\phi$ ranges over the two nontrivial orbits above. The exponents can be computed using the identity $S_0S_1S_2 = T^{-3}$, since the characteristic polynomial of $\Ind\phi(T)$ is $X^3-\phi(T^3)$.
\[
\begin{array}{c|c|c}
\textrm{Orbit} & \textrm{Exponents} & \textrm{Minimal weight}\\
\hline
\{\alpha_1\alpha_2,\alpha_0\alpha_2,\alpha_0\alpha_1\}& 0,~ 1/3,~ 2/3&2\\
\{\alpha_0,\alpha_1,\alpha_2\}&1/6,~1/2,~5/6&4\\
\end{array}
\]
The modular forms for $\alpha_1\alpha_2$, $\alpha_0\alpha_2$ and $\alpha_0\alpha_1$ can be found inside the space $M_2(\Gamma(3))$, which is a three-dimensional space of modular forms spanned by Eisenstein series. If $G_\alpha$ is the Eisenstein series of level $3$ associated to the cusp $\alpha$, then equation (4) in Section VII of \cite{Schoeneberg} shows that $G_\alpha|g = G_{g^t\alpha}$ for all $g \in \SL_2(\ZZ)$. Using this, one finds
\begin{align*}
  A &= G_0+G_1-G_2-G_\infty \in M_2(\alpha_0\alpha_1),\\
  B &= G_0-G_1+G_2-G_\infty \in M_2(\alpha_0\alpha_2),\\
  C &= G_0- G_1- G_2+G_\infty \in M_2(\alpha_1\alpha_2).
\end{align*}
These modular forms satisfy the quadratic relation
\[
  A^2+\zeta B^2-(\zeta+1)C^2 = 0,
\]
which is an example of the relation (\ref{equation:definingEquations}) from Section \ref{coherentSheaves}. In this case $\cL_2$ again has the correct degree of $1/2$ and we can give the automorphic description:
\[
  S(\Gamma^3) \cong \frac{\CC[A,B,C]}{(A^2+\zeta B^2-(\zeta+1)C^2)}.
\]
If $\rho$ denotes a representation of $\Gamma^3$ and $L$ denotes a choice of exponents for $\rho(T^3)$, then
\[
  \GM_L(\rho) \cong M_L(\rho) \oplus M_L(\rho\otimes \alpha_0\alpha_1)\oplus M_L(\rho\otimes \alpha_0\alpha_2)\oplus M_L(\rho\otimes \alpha_1\alpha_2).
\]
By Theorem \ref{t:geigleLenzing}, the $S(\Gamma^3)$-module $\GM_L(\rho)$ is maximal Cohen-Macaulay, although it no longer need be free.

Since the signature of $\sM_{\Gamma^3}$ is $(2,2,2)$, Theorem \ref{thm:ExistenceOfRank2Indec} states that there exists an indecomposable bundle of rank two on $\sM_{\Gamma^3}$. To give a modular description of such a bundle, consider the following family of representations $\rho_{z}$ indexed by $z \in \PP^1$:
\begin{align*}
  \rho_z(S_0) &= \twomat {-1}10{1}, & \rho_z(S_1) & = \twomat {-1}z0{1}, & \rho_z(S_2) &=\twomat {-1}00{1},
\end{align*}
if $z \neq \infty$, while if $z = \infty$ set
\begin{align*}
\rho_{\infty}(S_0) &= \twomat {-1}00{1}, & \rho_\infty(S_1) &= \twomat {-1}10{1}, & \rho_\infty(S_2) &= \twomat {-1}00{1}.
\end{align*}
These are all of the extensions of the trivial character by $\chi$, up to isomorphism as abstract representations (not as extensions). Note that $\Omega^1_{\sM_{\Gamma^3}} \cong \cV(\chi)$ so if $\rho=\rho_z$ is one of the above representations, then we get a corresponding extension of vector bundles 
$$
0 \rightarrow \Omega^1_{\sM_{\Gamma^3}}  \rightarrow \cV(\rho) \rightarrow \cO_{\sM_{\Gamma^3}} \rightarrow 0.
$$
As in the proof of Theorem \ref{thm:ExistenceOfRank2Indec}, the vector bundle $\cV(\rho)$ is indecomposable if and only if the connecting homomorphism $\delta$ in the long exact sequence
$$
0\rightarrow H^0(\sM_{\Gamma^3},\Omega^1_{\sM_{\Gamma^3}} )\rightarrow  H^0(\sM_{\Gamma^3},\cV(\rho))\rightarrow H^0(\sM_{\Gamma^3}, \cO_{\sM_{\Gamma^3}}) \stackrel{\delta}\rightarrow H^1(\sM_{\Gamma^3},\Omega^1_{\sM_{\Gamma^3}} ) \rightarrow \ldots 
$$
is nontrivial. Since $H^0(\sM_{\Gamma^3},\Omega^1_{\sM_{\Gamma^3}} )=0$ and $\dim H^0(\sM_{\Gamma^3},\cO_{\sM_{\Gamma^3}} )=1$, the nontriviality of $\delta$ is equivalent to $H^0(\sM_{\Gamma^3},\cV(\rho))= M_0(\Gamma^3,\rho)=  0$. One can show that there exists a unique $z_0$ such that $M_0(\Gamma^3,\rho_{z_0}) \neq 0$, and hence there exists a unique $z_0$ such that the bundle associated to $\rho_{z_0}$ decomposes as a direct sum of two line bundles. All other $z \in \PP^1$ yield indecomposable and isomorphic bundles of rank two. 

To compute the value $z_0 \in \PP^1$ corresponding to a split bundle explicitly, note that a form $\stwovec ab \in M_0(\Gamma^3,\rho_z)$ satisfies $b \in M_0(\Gamma^3,1)$, so that $b$ is constant, and $a' \in M_2(\Gamma^3,\chi)$. In particular, the derivative $a'$ is independent of $z \in \PP^1$. To determine $M_2(\Gamma^3,\chi)$, note first that it is not hard to see that $\dim M_2(\Gamma^3,\chi) = 1$. Hence since $M_2(\Gamma^3,\chi)$ clearly contains $\eta^4$, it is in fact spanned by this form. It follows that $a'$ is a scalar multiple of $\eta^4$. Hence after rescaling $\stwovec ab$, we may assume that $a$ is an antiderivative of $\eta^4$.

Using the transformation law for functions in $M_0(\Gamma^3,\rho_z)$, one sees that since $\rho_z(S_2)$ is diagonal, necessarily this antiderivative $a$ must vanish at $\tau_2 = i+2$, the fixed point of $S_2$. This pins $a$ down uniquely as the integral
\[
  a(\tau) = \int_{i+2}^\tau \eta^4(z)dz.
\]
Examination of the transformation law of $\stwovec ab$ at the other two elliptic points $\tau_0 = i$ and $\tau_1 = i+1$ shows that $z_0 = \frac{a(i+1)}{a(i)}$. Thus,
\[
  \frac{1}{z_0} = \frac{\int_{i+2}^{i}\eta^4d\tau}{\int_{i+2}^{i+1}\eta^4d\tau} = 1+\frac{\int_{i+1}^{i}\eta^4d\tau}{\int_{i+2}^{i+1}\eta^4d\tau} = 1 + \chi(T^{-1}) = 1 + e^{-2\pi i/6},
\]
so that if $z \neq \frac 16(3 +\sqrt{3}i)$, then $\cV(\rho_z)$ is indecomposable. For such $z$, the module $\GM(\rho_z)$ is maximal Cohen-Macaulay but not free over $S(\Gamma^3)$, by Proposition \ref{p:notfree}.

\subsection{A nonnormal subgroup of signature (2,2,3) }
The previous examples were normal subgroups of $\PSL_2(\ZZ)$ with cyclic quotient, which simplified some of the computations. Our final example concerns a nonnormal subgroup of index $4$. There is a well-known isomorphism $S_4 \cong \PSL_2(\ZZ/4\ZZ)$. Let $G$ be the subgroup of $\PSL_2(\ZZ)$ generated by $\Gamma(4)$ and the matrices $S = \stwomat 0{-1}10$ and $A = \stwomat 1112$. Then $G/\Gamma(4)$ is isomorphic with $S_3$ and the image of $G$ in $\PSL_2(\ZZ/4\ZZ)$ is equal to the following set of representative matrices:
\begin{align*}
\twomat 1001 & \equiv I& \twomat 1112 & \equiv A & \twomat 2331 & \equiv A^2\\
\twomat 0310  &\equiv S & \twomat 1233 &\equiv SA & \twomat 3121& \equiv SA^2
\end{align*}
This describes the subgroup $G$ of index $4$ in $\PSL_2(\ZZ)$ as a congruence subgroup. The elements $1$, $T$, $T^2$ and $T^3$ are a full set of nontrivial coset representatives of $G$ in $\PSL_2(\ZZ)$. Thus, the only possible elliptic points for $G$ are the $G$-orbits of $i$, $i+1$, $i+2$, $i+3$ and $\zeta$, $\zeta+1$, $\zeta+2$, $\zeta+3$, where $\zeta = e^{2\pi i/3}$. By computing the full stabilizers of these points in $\PSL_2(\ZZ)$ and then considering the congruence description of $G$, one easily sees that the elliptic points for $G$ are $\tau_0 = i$, $\tau_1 = i+3$ and $\tau_2 = \zeta+2$. Thus $G$ is of genus zero with a single cusp of width $4$ and three elliptic points of signature $(2,2,3)$. By Proposition \ref{p:fuchsianpresentation}, $G$ is generated freely by the matrices
\begin{align*}
S_0 &= S, & S_1 &= T^3ST^{-3}, & R_2 &= T^2RT^{-2} = T^2ST^{-1}.
\end{align*}

Define characters of $G$ by setting $\alpha_0(S_0) = -1$, $\alpha_0(S_1) = \alpha_0(R_2) = 1$. Define $\alpha_1$ similarly but with $S_0$ and $S_1$ permuted, and define $\alpha_2$ by setting $\alpha_2(R_2) = \zeta$, $\alpha_2(S_0) = \alpha_2(S_1) = 1$. These characters generate the character group $\Hom(G,\CC^\times)$. If $\chi$ is the usual character of $\PSL_2(\ZZ)$ then $\chi(S) = -1$ and $\chi(R) = \zeta^2$. Thus since $G$ contains elliptic elements of order $2$ and $3$, its restriction to $G$ remains order $6$. In fact, we have $\chi = \alpha_0\alpha_1\alpha_2^2$ as characters of $G$. Since $T^4 = S_1R_2S_0$, and a character $\phi$ for $G$ is cupsidal if and only if $\phi(T^4) = 1$, the cuspidal characters are $1$ and $\alpha_0\alpha_1 = \chi^3$. Unfortunately in this case $\cL_2$ has degree 2/3, so $\cL_2$ does not generate the non-torsion part of the Picard group, since $m = \lcm(2,2,3) = 6$. Therefore tensoring $\cL_2$ with all the cuspidal characters of $G$ fails to capture all of the ring $S(G)$ of geometrically weighted modular forms for $G$, and so we cannot give a simple automorphic description of $S(G)$ as we did in the previous examples.

In order to determine the Hilbert-Poincar\'{e} polynomials of the characters $\chi^n$ and to describe the modules $M(G,\chi^n)$, observe that the trivial representation of $G$ induces to the standard representation of $S_4 \cong \PSL_2(\ZZ)/\Gamma(4)$, and so it decomposes into a one-dimesional trivial character and a three dimensional irreducible representation. Since $T$ corresponds to the cycle $(1234)$ in this representation, the exponents for the three-dimensional irrep are $1/4$, $1/2$ and $3/4$. Hence $\Tr(L) =  3/2$, and for three-dimensional irreps the minimal weight is $4\Tr(L)-2$, which in this case is $4$. Thus the Hilbert-Poincar\'{e} polynomial for the trivial character of $G$ is $\frac{1+T^4+T^6+T^8}{(1-T^4)(1-T^6)}$. This is the Hilbert-Poincar\'{e} polynomial for the ring of classical scalar modular forms for $G$. More generally, twisting by $\chi$ and using the formulae of Section 6 of \cite{CandeloriFranc} yields the following Hilbert-Poincar\'{e} polynomials for the powers of $\chi$:
\begin{align*}
1 :& \quad \frac{1 + T^4 + T^6 + T^8}{(1-T^4)(1-T^6)} & \chi :&\quad \frac{T^2 + T^6 + T^8 + T^{10}}{(1-T^4)(1-T^6)} & \chi^2 :&\quad \frac{2T^4 + T^6 + T^8}{(1-T^4)(1-T^6)} \\
 \chi^3 :&\quad \frac{T^2+ T^4 + 2T^6}{(1-T^4)(1-T^6)} & \chi^4 :&\quad \frac{T^4 + T^6 + 2T^8}{(1-T^4)(1-T^6)} & \chi^{5} :&\quad \frac{T^2 + T^4 + T^6 + T^{10}}{(1-T^4)(1-T^6)} 
\end{align*}
In order to describe a rank two indecomposable bundle for $G$, we first classify all indecomposable representations of rank $2$. The classification of such representations containing the trivial representation is as follows -- the general case can be deduced from this by tensoring with a character.
\begin{prop}
\label{prop:indecomposableRepsOfG}
Up to isomorphism of $\rho$, the following lists all of the nonsplit indecomposable rank $2$ representations of $G$ containing a copy of the trivial representation: 
\begin{enumerate}
\item there is a single extension
\[
  0 \to 1 \to \rho \to \chi^{3} \to 0
\]
given by
\begin{align*}
\rho(S_0) &= \twomat{1}{1}{0}{-1} & \rho(S_1) &= \twomat{1}{0}{0}{-1} & \rho(R_2) &= \twomat{1}{0}{0}{1}. 
\end{align*}
\item there are two infinite families of nonisomorphic representations parameterized by $\PP^1$ which arise from extensions
\[
  0 \to 1 \to \rho \to \chi^{-a} \to 0
\]
when $a = 1$ or $5$, given by
\begin{align*}
\rho(S_0) &= \twomat{1}{1}{0}{-1} & \rho(S_1) &= \twomat{1}{z}{0}{-1} & \rho(R_2) &= \twomat{1}{0}{0}{\zeta^a}\\ 
\rho(S_0) &= \twomat{1}{0}{0}{-1} & \rho(S_1) &= \twomat{1}{1}{0}{-1} & \rho(R_2) &= \twomat{1}{0}{0}{\zeta^a}
\end{align*}
where $z \in \CC$.
\item there are two extensions
\[
  0 \to 1 \to \rho \to \alpha_0\chi^{-a}\to 0
\]
where $a = 2$ or $4$, given by
\begin{align*}
\rho(S_0) &= \twomat{1}{1}{0}{-1} & \rho(S_1) &= \twomat{1}{0}{0}{1} & \rho(R_2) &= \twomat{1}{0}{0}{\zeta^a}. 
\end{align*}
\item there are two extensions
\[
0\to 1 \to \rho \to \alpha_0\chi^{-a} \to 0
\]
when $a = 1$ or $5$ given by
\begin{align*}
\rho(S_0) &= \twomat{1}{0}{0}{1} & \rho(S_1) &= \twomat{1}{1}{0}{-1} & \rho(R_2) &= \twomat{1}{0}{0}{\zeta^a}. 
\end{align*}
\end{enumerate}
\end{prop}
\begin{proof}
This is a standard cohomological computation. Given an extension
\[
 0 \to 1 \to \rho \to \phi \to 0
\]
for some character $\phi$ of $G$, let $\kappa$ denote the top right entry of $\rho$. This is a $1$-cocycle living in
\[
 Z^1 = \{\kappa \colon G \to \CC \mid \kappa(gh) = \phi(h)\kappa(g)+\kappa(h)\}.
\]
The isomorphism classes of extensions of $\phi$ by $1$ are parameterized by the cohomology group $H^1 = Z^1/B^1$ where 
\[B^1 = \{\kappa \in Z^1 \mid \exists z \in \CC \textrm{ such that } \kappa(g) = (\phi(g)-1)z\}.\]
Being isomorphic as extensions is stricter than being isomorphic as abstract representations. The isomorphy classes as abstract representations arising from the nontrivial extensions are identified with the projective space $\PP(H^1)$. We will describe the computation for the extensions of $\phi = \chi^{-a}$ by $1$ and omit the details for the other six characters.

There is an embedding $f \colon Z^1 \to \CC^3$ defined by $f(\kappa) = (\kappa(S_0),\kappa(S_1),\kappa(R_2))$. Observe that
\[
  f(B^1) = \{(((-1)^{a}-1)z,((-1)^{a}-1)z, (\zeta^{a}-1)z) \mid z \in\CC\}.
\]
In particular, $B^1$ is zero dimensional if $a = 0$ and otherwise it is one dimensional. To determine $Z^1$, let us first examine what the cocycle conditions enforces. Of course $\kappa(1) = 0$. Since $S_0^2 = S_1^2 = R_2^3 = 1$ we deduce that
\[
0=((-1)^a+1)\kappa(S_1) =  ((-1)^a+1)\kappa(S_2) =  (\zeta^{2a}+\zeta^a+1)\kappa(U)
\]
If $a = 0$ then $\kappa = 0$ and $Z^1$ is also zero dimensional. If $a =  2,4$ then $\kappa(S_0) = \kappa(S_1) = 0$ and $Z^1$ is at most one dimensional. In this case there are no nontrivial extensions. If $a = 3$ then $\kappa(R_2) = 0$ and $\kappa(S_0)$ and $\kappa(S_1)$ can be nonzero. Hence after modding out by $B^1$ we get a one dimensional space of extensions. Finally if $a = 1,5$ then there are no conditions whatsoever and $Z^1$ is three dimensional, hence there is a two dimensional space of extensions. Since the nontrivial extensions break up into isomorphism classes according to $\PP(Z^1/B^1)$, one deduces the stated classification.
\end{proof}

Let $\rho$ denote an indecomposable but not irreducible rank two representation of the form
$$
0 \rightarrow \chi^{5} \rightarrow \rho \rightarrow 1 \rightarrow 0.
$$
As with the subgroup of index three, the corresponding bundle $\cV(\rho)$ on $\sM_G$ is decomposable if and only if there exist nonzero modular forms for $\rho$ of weight zero. To write down such a form explicitly, note that by Proposition \ref{prop:indecomposableRepsOfG}, up to isomorphism $\rho$ must be of the form $\rho_z$ for $z\in \CC$, or $\rho_{\infty}$, where
\begin{align*}
\rho_z(S_1) &= \twomat{-1}{-1}{0}{1} & \rho_z(S_2) &= \twomat{-1}{z}{0}{1} & \rho_z(R_2) &= \twomat{\zeta^2}{0}{0}{1}\\ 
\rho_\infty(S_1) &= \twomat{-1}{0}{0}{1} & \rho_\infty(S_2) &= \twomat{-1}{-1}{0}{1} & \rho_\infty(R_2) &= \twomat{\zeta^2}{0}{0}{1}. 
\end{align*}
Let $F = \stwovec a b$ be nonconstant of weight zero for one of the representations $\rho_z$. As above, $b$ must be a nonzero scalar and $a' \in M_2(G,\chi^5)$. Since we computed the Hilbert-Poincar\'{e} polynomial of $\chi^5$ on $G$ to be 
\[
\sum_{k \in \ZZ} \dim M_k(G,\chi^5)T^k = \frac{T^2+T^4+T^6+T^{10}}{(1-T^4)(1-T^6)}
\]
it follows that $M_2(G,\chi^5)$ is one dimensional.

To describe $M_2(G,\chi^{5})$, observe that the induction $\Ind \chi^{5}$ to $\PSL_2(\ZZ)$ is a four dimensional representation that breaks up into a copy of $\chi^{5}$ regarded as a character of $\PSL_2(\ZZ)$, and a three dimensional irreducible representation with exponents $1/12$, $1/3$ and $7/12$. It is only the three dimensional irreducible representation that contributes to $M_2(G,\chi^5)$. Using the results of \cite{FrancMason2} on three dimensional representations of $\PSL_2(\ZZ)$, one can show that $M_2(G,\chi^{5})$ is spanned by a linear combination of forms with Fourier coefficients:
\begin{align*}
  f_1 &= q^{\frac 1{12}}\left(1 + 2q-5q^2-10q^3+9q^4 + 14q^5 -10q^6+14q^8+ \cdots\right),\\
  f_2 &= q^{\frac{1}{3}}\left(1-4q^2+2q^4+8q^6-5q^8+ \cdots\right),\\
  f_3 &= q^{\frac{7}{12}}\left(1-2q+q^2-2q^3+4q^5+\cdots\right).
\end{align*}
One finds that
\begin{align*}
f_1 &= \left(\frac{\eta(q^2)^5}{\eta(q)\eta(q^4)^2}\right)^2, & f_2 &= \eta(q^2)^4, & f_3 &= \left(\frac{\eta(q)\eta(q^4)^2}{\eta(q^2)}\right)^2.
\end{align*}
Using the transformation law for $\eta^2$, it is not hard to show that $f = f_1+4f_3$ spans $M_2(G,\chi^{5})$. Hence, after possibly rescaling $\stwovec ab \in M_0(G,\rho_{z})$, we may assume that $a' = f$. As in the case of the group of index $3$, since $\rho_z(R_2)$ is diagonal, we must have
\[
  a(\tau) = \int_{\zeta+2}^\tau f(z)dz.
\]
If $a(\tau_0) = a(\tau_1) = 0$ then one easily checks from the transformation law for $a$ that it is in fact a scalar modular form of weight zero for $\chi^{5}$ on $G$, hence it must be constant. Since $a$ is not constant, at least one of $a(\tau_0)$ or $a(\tau_1)$ is nonzero. The transformation law for $a$ then shows that the point $z_0 = (a(\tau_0):a(\tau_1)) \in \PP^1$ corresponds to the unique representation $\rho_{z_0}$ yielding the decomposable bundle, and all other bundles $\cV(\rho_z)$ are indecomposable. We used \emph{Sage} to compute this value numerically:
\[
  z_0 = \frac{a(i+3)}{a(i)} = 1.0910849089\ldots + 0.4942818186\ldots i.
\]
Computer experiments do not suggest that $z_0$ is algebraic. This may be because it is obtained by integrating a classical scalar congruence form between CM points corresponding to different CM fields, in this case $\QQ(i)$ and $\QQ(\zeta)$, and so the two different CM periods intervene.\footnote{In the case of the subgroup of index three, all three elliptic points are for the same CM field $\QQ(i)$, and we obtained an algebraic value for $z_0$ in that case.} Note that if $a$ were an ordinary scalar modular form, and not an antiderivative of such a form, then of course this ratio of CM values would be algebraic by \cite{Shimura}. These computations suggest that there may be some interesting arithmetic encoded in the CM values of vector valued modular forms for some nonunitarizable representations of congruence groups.

\renewcommand\refname{References}
\bibliographystyle{alpha}
\bibliography{Tilting}

\begin{thebibliography}{FM14b}

\bibitem[Ati57]{Atiyah}
M.~F. Atiyah.
\newblock Vector bundles over an elliptic curve.
\newblock {\em Proc. London Math. Soc. (3)}, 7:414--452, 1957.

\bibitem[BH16]{BiswasHogadi}
Indranil Biswas and Amit Hogadi.
\newblock Unitary representations of the fundamental group of orbifolds.
\newblock {\em Proc. Indian Acad. Sci. Math. Sci.}, 126(4):557--575, 2016.

\bibitem[Bor92]{Borcherds}
Richard~E. Borcherds.
\newblock Monstrous moonshine and monstrous {L}ie superalgebras.
\newblock {\em Invent. Math.}, 109(2):405--444, 1992.

\bibitem[CB10]{Crawley-Boevey}
William Crawley-Boevey.
\newblock Kac's theorem for weighted projective lines.
\newblock {\em J. Eur. Math. Soc. (JEMS)}, 12(6):1331--1345, 2010.

\bibitem[CF16]{CandeloriFranc}
Luca Candelori and Cameron Franc.
\newblock Vector valued modular forms and the modular orbifold of elliptic
  curves.
\newblock {\em Int. J. of Num. Th.}, 2016.

\bibitem[CN79]{ConwayNorton}
J.~H. Conway and S.~P. Norton.
\newblock Monstrous moonshine.
\newblock {\em Bull. London Math. Soc.}, 11(3):308--339, 1979.

\bibitem[FM14a]{FrancMason1}
Cameron Franc and Geoffrey Mason.
\newblock Fourier coefficients of vector-valued modular forms of dimension 2.
\newblock {\em Canad. Math. Bull.}, 57(3):485--494, 2014.

\bibitem[FM14b]{FrancMason2}
Cameron Franc and Geoffrey Mason.
\newblock Hypergeometric series, modular linear differential equations and
  vector-valued modular forms.
\newblock {\em The Ramanujan Journal}, pages 1--35, 2014.

\bibitem[Gan06]{Gannon:book}
T.~Gannon.
\newblock {\em Moonshine beyond the {M}onster}.
\newblock Cambridge Monographs on Mathematical Physics. Cambridge University
  Press, Cambridge, 2006.
\newblock The bridge connecting algebra, modular forms and physics.

\bibitem[Gan14]{Gannon}
Terry Gannon.
\newblock The theory of vector-valued modular forms for the modular group.
\newblock In W.~Kohnen and R.~Weissauer, editors, {\em Conformal Field Theory,
  Automorphic Forms and Related Topics: CFT, Heidelberg, September 19-23,
  2011}, Contributions in Mathematical and Computational Sciences, pages
  247--286. Springer Berlin Heidelberg, 2014.

\bibitem[GL87]{GeigleLenzing}
Werner Geigle and Helmut Lenzing.
\newblock A class of weighted projective curves arising in representation
  theory of finite-dimensional algebras.
\newblock In {\em Singularities, representation of algebras, and vector bundles
  ({L}ambrecht, 1985)}, volume 1273 of {\em Lecture Notes in Math.}, pages
  265--297. Springer, Berlin, 1987.

\bibitem[Gro57]{Groth}
A.~Grothendieck.
\newblock Sur la classification des fibr\'es holomorphes sur la sph\`ere de
  {R}iemann.
\newblock {\em Amer. J. Math.}, 79:121--138, 1957.

\bibitem[Kac90]{Kac}
Victor~G. Kac.
\newblock {\em Infinite-dimensional {L}ie algebras}.
\newblock Cambridge University Press, Cambridge, third edition, 1990.

\bibitem[Len86]{Lenzing}
Helmut Lenzing.
\newblock Curve singularities arising from the representation theory of tame
  hereditary algebras.
\newblock In {\em Representation theory, {I} ({O}ttawa, {O}nt., 1984)}, volume
  1177 of {\em Lecture Notes in Math.}, pages 199--231. Springer, Berlin, 1986.

\bibitem[Mas08]{Mason}
Geoffrey Mason.
\newblock 2-dimensional vector-valued modular forms.
\newblock {\em Ramanujan J.}, 17(3):405--427, 2008.

\bibitem[Mel04]{Meltzer:MemAMS}
Hagen Meltzer.
\newblock Exceptional vector bundles, tilting sheaves and tilting complexes for
  weighted projective lines.
\newblock {\em Mem. Amer. Math. Soc.}, 171(808):viii+139, 2004.

\bibitem[MM10]{MarksMason}
Christopher Marks and Geoffrey Mason.
\newblock Structure of the module of vector-valued modular forms.
\newblock {\em J. Lond. Math. Soc. (2)}, 82(1):32--48, 2010.

\bibitem[NS64]{NarasimhanSeshadri}
M.~S. Narasimhan and C.~S. Seshadri.
\newblock Stable bundles and unitary bundles on a compact {R}iemann surface.
\newblock {\em Proc. Nat. Acad. Sci. U.S.A.}, 52:207--211, 1964.

\bibitem[Poi82]{Poincare1}
H.~Poincare.
\newblock M\'emoire sur les fonctions fuchsiennes.
\newblock {\em Acta Math.}, 1(1):193--294, 1882.

\bibitem[Poi84]{Poincare2}
H.~Poincar\'e.
\newblock M\'emoire sur les fonctions z\'etafuchsiennes.
\newblock {\em Acta Math.}, 5(1):209--278, 1884.

\bibitem[Ran77]{Rankin}
Robert~A. Rankin.
\newblock {\em Modular forms and functions}.
\newblock Cambridge University Press, Cambridge-New York-Melbourne, 1977.

\bibitem[Sch74]{Schoeneberg}
Bruno Schoeneberg.
\newblock {\em Elliptic modular functions: an introduction}.
\newblock Springer-Verlag, New York-Heidelberg, 1974.
\newblock Translated from the German by J. R. Smart and E. A. Schwandt, Die
  Grundlehren der mathematischen Wissenschaften, Band 203.

\bibitem[Ser55]{SerreFAC}
Jean-Pierre Serre.
\newblock Faisceaux alg\'ebriques coh\'erents.
\newblock {\em Ann. of Math. (2)}, 61:197--278, 1955.

\bibitem[Shi75]{Shimura}
Goro Shimura.
\newblock On some arithmetic properties of modular forms of one and several
  variables.
\newblock {\em Ann. of Math. (2)}, 102(3):491--515, 1975.

\bibitem[Sim91]{Simpson1}
Carlos~T. Simpson.
\newblock Nonabelian {H}odge theory.
\newblock In {\em Proceedings of the {I}nternational {C}ongress of
  {M}athematicians, {V}ol.\ {I}, {II} ({K}yoto, 1990)}, pages 747--756. Math.
  Soc. Japan, Tokyo, 1991.

\bibitem[Sim11]{Simpson2}
Carlos Simpson.
\newblock Local systems on proper algebraic {$V$}-manifolds.
\newblock {\em Pure Appl. Math. Q.}, 7(4, Special Issue: In memory of Eckart
  Viehweg):1675--1759, 2011.

\bibitem[Sti92]{Stillwell}
John Stillwell.
\newblock {\em Geometry of surfaces}.
\newblock Universitext. Springer-Verlag, New York, 1992.
\newblock Corrected reprint of the 1992 original.

\bibitem[VZB]{VoightZB}
John Voight and David Zureick-Brown.
\newblock The canonical ring of a stacky curve.
\newblock {\em arXiv:1501.04657}.

\end{thebibliography}
\end{document}